\documentclass[11pt, 3p, times, reqno]{article}
\usepackage[right=2.4 cm,left=2.4 cm,top=2.3 cm,bottom=2.3 cm]{geometry}
\usepackage[utf8]{inputenc}
\usepackage{amsmath,amsxtra,amssymb,latexsym,amscd,amsthm,mathrsfs,amsfonts,amstext,amsbsy,upgreek,dsfont,eufrak,float,cite}
\usepackage{fancyhdr,booktabs,verbatim}
\usepackage[unicode,colorlinks,citecolor=blue]{hyperref}
\usepackage{multicol,graphicx,array,multirow,fancyhdr,booktabs,xcolor}
\usepackage{bbm,csquotes}
\usepackage{extarrows}
\usepackage{esint,enumerate}

\newcommand {\norm}[1] {\left\Vert #1 \right\Vert}
\newcommand {\bnorm}[1] {\big\Vert #1 \big\Vert}
\newcommand {\Bnorm}[1] {\Big\Vert #1 \Big\Vert}
\newcommand {\Bgnorm}[1] {\Bigg\Vert #1 \Bigg\Vert}

\newcommand {\binner}[2] {\big\langle #1,#2 \big\rangle}
\newcommand {\Binner}[2] {\Big\langle #1,#2 \Big\rangle}

\def\dis{\displaystyle}
\def\N{\mathbb{N}}
\def\R{\mathbb{R}}
\def\sumuse{\sum_{k=1}^{\infty}}
\def\L2{L^2(\Omega)}

\def\Lxi{L^{\Xi}(\Omega)}

\def\d{\mathrm{d}}

\def\Dnu{D^{\nu}(\Omega)}
\def\D1{D^{1}(\Omega)}
\def\Doneminus{D^{-1}(\Omega)}
\def\Dnuminus{D^{\nu-1}(\Omega)}

\newtheorem{theorem}{{\bf Theorem}}[section]
\theoremstyle{definition} \newtheorem{definition}[theorem]{\bf Definition}

\theoremstyle{plain} \newtheorem{lemma}[theorem]{Lemma}

\newtheorem{remark}{Remark}[section]

\title{\textbf{Global well-posedness for fractional Sobolev-Galpern type equations}}
\date{}
\author{ Huy Tuan Nguyen $^{\mathrm {a,b}}$, Nguyen Anh Tuan $^{\mathrm{c}}$, Chao Yang$^{\mathrm d,}$\footnote{Corresponding author: \url{yangchao_@hrbeu.edu.cn} (Chao Yang). Other authors: \url{nhtuan@hcmus.edu.vn} (Huy Tuan Nguyen), \url{nguyenanhtuan@tdmu.edu.vn} (Nguyen Anh Tuan).}  \\\\
	\small $^{\mathrm a}$Department of Mathematics and Computer Science, 
	University of Science, 
	Ho Chi Minh City, 
	Vietnam\\
	\small $^{\mathrm b}$ Vietnam National University, 
	Ho Chi Minh City, 
	Vietnam\\
	\small $^{\mathrm c}$ Division of Applied Mathematics, 
	Thu Dau Mot University, 
	Binh Duong Province, 
	Vietnam\\
	\small $^{\mathrm d}$College of Mathematical Sciences, Harbin Engineering University, 150001, People's Republic of China 	
}

\allowdisplaybreaks
\begin{document}
\mdseries 
\maketitle	
\begin{abstract}
This article is a comparative study on an initial-boundary value problem for a class of semilinear pseudo-parabolic equations with the fractional Caputo derivative, also called the fractional Sobolev-Galpern type equations. The purpose of this work is to reveal the influence of the degree of the source nonlinearity on the well-posedness of the solution. By considering four different types of nonlinearities, we derive the global well-posedness of mild solutions to the problem corresponding to the four cases of the nonlinear source terms. For the advection source function case, we apply a nontrivial limit technique for singular integral and some appropriate choices of weighted Banach space to prove the global existence result. For the gradient nonlinearity as a local Lipschitzian, we use the Cauchy sequence technique to show that the solution either exists globally in time or blows up at finite time. For the polynomial form nonlinearity, by assuming the smallness of the initial data we derive the global well-posed results. And for the case of exponential nonlinearity in two-dimensional space, we derive the global well-posedness by additionally use of Orlicz space.

\vspace*{0.1cm}
\end{abstract}

\section{Introduction}
In the current work, we especially concern about the following time-fractional pseudo-parabolic equation
\begin{align}\label{Main Equation}
\partial_{t}^{\alpha}u - \partial_{t}^{\alpha}\Delta u -\Delta u =H\left(u\right),
\end{align}
wherein $\Omega$ is a bounded domain of $\mathbb{R}^{d}~(d\in\N), u$ is an unknown function from $\Omega\times(0, \infty)$ to $\mathbb{R}$, $H$ is the source function which is going to be  defined in more details later. The notation $\partial_{t}^{\alpha}$ is abbreviated for the time-fractional derivative of order $\alpha\in(0,1)$ in the Caputo sense,  defined by
\begin{align*}
\partial_{t}^{\alpha}u(t) =\int_{0}^{t}\frac{(t-\tau)^{\alpha-1}}{\Gamma(\alpha)}\partial_{\tau}u(\tau)\d\tau,
\end{align*}
provided that the right-hand side (RHS) of the above equation makes sense. In addition, for Equation \eqref{Main Equation}, we assume that the boundary $ \partial\Omega $ of $\Omega$ is sufficiently smooth. Also, for Equation \eqref{Main Equation} we consider the homogeneous Dirichlet boundary condition
\begin{align}\label{Dirichlet condition}
u(x, t) =0,\quad (x, t) \in\partial\Omega\times(0, \infty),  
\end{align}
and the following initial value condition
\begin{align}\label{Initial condition}
u(x,t)=u_{0}(x), \quad (x, t) \in\Omega\times\{0\},
\end{align}
here, $u_{0}$ is an initial function which satisfies some specific assumptions in different theorems claimed in the present paper. In Equation \eqref{Main Equation}, the operator $(-\Delta)=\mathcal{A}: \text{dom}(\mathcal{A})\subset L^{2}(\Omega)\rightarrow L^{2}(\Omega)$ is uniformly symmetric elliptic. As we know that, $\mathcal{A}$ owns a set of positive eigenvalues $\{\theta_{k}\}_{k\in \mathbb{N}}$ whose elements organize an increasing sequence. 
Corresponding to this set of Dirichlet eigenvalues, there is a set of eigenfunctions $\{\phi_{k}\}_{k\in \mathbb{N}}$ which forms a complete orthonormal of $L^{2}(\Omega)$. Our main goal is to investigate the globally well-posed results for a mild solution of the initial-boundary value problem \eqref{Main Equation}-\eqref{Initial condition}, and we shall introduce the background of our work in the following subsection.

\subsection{Background of the problem}
A generalized form of Equation \eqref{Main Equation} reflecting the theory of isotropic incompressible homogeneous fluid was first considered by Coleman and Walter in \cite{Coleman1}. To investigate some non-steady fluid flows as the second-order fluids due to the effect of pressure, Ting in \cite{Ting app1} considered a special type semilinear pseudo-parabolic equation in the following form
\begin{align}\label{FluidEquation}
a\partial_{t}u=b\Delta u+c\partial_{t}\Delta u+H(u).
\end{align}
This equation along with a specified type of initial-boundary conditions describes a bounded flow whose external force affects its solid boundary. And such types of equations like \eqref{FluidEquation} also appear in the study about thermodynamic temperature \cite{Chen pseu app2} and population recovery \cite{Padron pseu app3}. In one-dimensional space, by taking $ b=0 $ and $ H(u)=-u_x-uu_x $ in \eqref{FluidEquation}, we have the regularized long-wave equation or the Benjamin–Bona–Mahony BBM equation, which was proposed by Benjamin, Bona, and Mahony in \cite{Benjamin} with applications in the study of long waves propagation. The work \cite{Benjamin} was extended by Amick, Bona, and Schonbek in \cite{AmickBBM} by investigating the large time behavior of solution to the Cauchy problem related to the one-dimensional case of \eqref{FluidEquation} with $ H(u)=-u_x-uu_x $. This nonlinear pseudo-parabolic equation can be seen as an additional consideration of dissipation mechanisms for the BBM equation, also  Celebi et.al. in \cite{BBM1} considered the generalized BBM equation (GBBM) in the form of   
\begin{align*}
	\partial_{t}u-\partial_{t}\Delta u-b\Delta u+(\eta\cdot\nabla)u+\nabla\cdot F(u)=0,
\end{align*}
where $ \eta\in\R^d $ is a constant vector, $ F(u) $ is a $ d-$dimentional vector field, which received a lot of attentions in the PDEs forum \cite{BBM2,BBM3,BBM4}. In addition, from \cite{Xu pseudo1,semiconductors,Xu pseudo3,5th comment} and the references given there, the readers will find that Equation \eqref{FluidEquation} can also be used in biological sciences, ﬁltration theory or the study of semiconductors.

Given the representativeness and application value of this kind of mathematical model, it has attracted many mathematicians' attention with many rich results. It is an impossible task to mention all of them, so we only make an overview of the works closely related to the research of this paper. In \cite{Ting pseudo}, Showalter and Ting considered an initial value problem for a generalized form of \eqref{FluidEquation} as follows
\begin{align*}
M\partial_{t} u+Lu=H,
\end{align*}
here, $ M, L $ are second-order differential operators independent of $ t $, containing variable coefficients with some specific properties, $ M $ is uniformly strongly elliptic on a bounded open set $\Omega\subset\R^d $. In the homogeneous case, based on properties of the Friedrichs extensions of $ M $ and $ L $ which can be obtained by the Lax-Milgram theorem, the authors constructed a group $ \left\{E(t):t\in\R\right\} $ which helps to get the unique existence of the weak solution. Also, thanks to the results that $ H^1_0(\Omega)\cap H^p $ is invariant under effects of the group $ \left\{E(t):t\in\R\right\} $, the regularity of the solution was proved. The main results of this work also include the asymptotic behavior for the solution and the extended theories for the nonhomogeneous case. We also notice that the different forms of the nonlinearities in \eqref{FluidEquation} are of special interest, and attract a lot of attention. In \cite{Xu pseudo1}, Xu and Su considered a pseudo-parabolic 
equation with the well-known polynomial source 
\begin{align*}
H(u)=u^p,\quad\text{where}\quad p\in(1,\infty) \text{ if } d=1,2,~\text{or}~ p\in\left(1,\frac{d+2}{d-2}\right) \text{ if } d\ge3.
\end{align*}
In consideration of the subcritical (resp. critical) case that the initial energy is less than (resp. equal to) the depth of potential well $ J(u_0)< \mathbf{d} $ (resp. $ J(u_0)= \mathbf{d} $) and the positive (resp. non-negative) value at $ u_0 $ of the Nehari functional, i.e, $ I(u_0)>0 $  (resp. $ I(u_0)\ge0 $), by using the Galerkin and the potential well theory, the authors proved the global existence, uniqueness and the asymptotic behavior of the solution. Otherwise, when $ I(u_0)<0 $, the solution is proved to be blowing up at a finite time. Further for the arbitrarily initial energy, that is $ J(u_0)>0 $, the comparison principle and variational methods are adopted to obtain the finite-time blow-up results. And the blowup time for such high energy case was also estimated in \cite{4th comment}. In \cite{tomasA,tomasB}, besides studying  \eqref{FluidEquation} with a source term satisfying growth conditions of polynomial type, Caraballo and his colleagues also considered the inﬂuence of external forces with some kind of delay. Namely, $ H $ was given by $ H(t,u)=f(t,u_t)+g(u) $, where $ f(t,u_t) $ is the time-dependent delay term caused by memory or hereditary characteristics and $ g\in C^1(\R) $ satisfying $  $
\begin{align*}
\limsup_{|v|\to+\infty}\frac{g(v)}{v}\le\frac{\theta_{1}}{6}\quad\text{ or }\quad\limsup_{|v|\to+\infty}\frac{g(v)}{v}\le0
\end{align*}
and
\begin{align*}
\big|g(v)-g(w)\big|\le C|v-w|\left(1+|v|^{p-1}+|w|^{p-1}\right),\quad p>1,C>0.
\end{align*}
The above series of work has aroused great interest among colleagues. On the one hand, there are a large number of practical problems surrounding this type of model; on the other hand, there is a huge gap between the existing research and the many different variants that exist widely. Therefore, a large number of subsequent researches are rapidly developed around different variants of the model, including the couple form of a parabolic system \cite{3rd comment}, the pseudo-parabolic model \eqref{FluidEquation} with the singular potential \cite{Xu pseudo2}, a nonlocal form of (4) with the nonlinearity $H(u)=|u|^{p-1}u-\fint_{\Omega}|u|^{p-1}u{\d} x$ \cite{Xu pseudo3}, the nonlocal version with the conical degeneration by considering the Fuchsian type Laplace operator \cite{5th comment}. In \cite{EiC1}, J. Zhou used the $ |x|^{\sigma} $ weighted $ L^{p+1}(\Omega) $ to study the global solutions and blow-up solutions to \eqref{FluidEquation} with $ H(u)=|x|^{\sigma}|u|^{p-1}u$ ($\sigma<0$). It's also necessary to consider the case where the nonlinear source function grows much faster than the polynomial level. In this case, the nonlinearity of exponential type is considered as an optimal alternative to the polynomial source. In \cite{11th comment} Zhu et.al. investigated Equation \eqref{FluidEquation} with $ H $ as an exponential nonlinearity. By the elliptic theory, the authors showed the local existence and uniqueness of the solution. And once again, the potential well theory was applied to prove that when the initial energy is low, this solution exists globally. A sufficient condition for a blowing-up solution without any limit of initial energy was also provided. The logarithmic forms of $ H $ were concerned in \cite{ChenTianPseudo,7th comment,pseudo log}. 
In \cite{ChenTianPseudo}, Chen and Tian studied the subcritical and critical energy cases for \eqref{FluidEquation} with $ H(u)=u\log|u| $. When $I(u_0)<0 $, the global existence and uniqueness of solutions were proved. In contrary, if $ I(u_0)<0 $,  unlike the polynomial cases mentioned above, the authors showed that the solution doesn't blow up in ﬁnite time but at $ +\infty $. The same topic was also concerned for an initial-boundary value problem for inﬁnitely degenerate semilinear pseudo-parabolic equations with logarithmic nonlinearity in \cite{7th comment}. The global existence and the asymptotic behavior of the solutions were discussed for the cases of subcritical/critical initial energy, and the infinite time blow up was also showed. A periodic logarithmic nonlinearity given by $ H(u)=m(t)u\log|u| $ for \eqref{FluidEquation} was concerned in \cite{pseudo log} by Ji et.al. The existence of the solution was established to show the instability further. When $ c=0 $, \eqref{FluidEquation} becomes the usual parabolic equation, which has been investigated a lot over the years, we refer the reader to \cite{EiC2,EiC3,Weissler} and the references therein.

The above work is a part of the representative work on the pseudo-parabolic equations, but in fact, there are many results on such mathematical models, and we obviously cannot list them all. Taking a glimpse of the whole leopard, we can still see that this type of model is widely and intensively concerned not only because of its physical and practical application background, but also because of its interesting mathematical phenomena. In particular, the above work has given us such an inspiration: the nonlinearity of the model dramatically affects the properties and behaviors of the solution. Roughly speaking, a weaker nonlinearity will ensure the global-in-time existence of the solution, while a stronger nonlinearity will cause the dynamic properties of the solution to be differentiated due to the scale of the initial data. We hope to systematically describe this phenomenon in a unified work, which is the original intention of this work. To achieve this goal, we have adopted the strategy of gradually strengthening the degree of the nonlinearity, that is, discussing the linear advective form inhomogeneous terms, and the inhomogeneous terms strengthened to the squared nonlinearity, the power-type nonlinearity, as well as the exponential nonlinear terms.

At the same time, we further expand our understanding of this type of problem from the perspective of the equation structure, that is, we consider the time-fractional derivative in the sense of Caputo for Equation \eqref{Main Equation}. In recent decades, the fractional calculus has been proved that it is useful in application to many fields of science such as the study of Brownian motion \cite{fractional Brownian motion}, chemical stimuluses of an organism with memory effects \cite{Keller Segel,Keller Segel sys}, and waves in linear viscoelastic media \cite{Podlubny}. The greatest motivation for considering Problem \eqref{Main Equation}-\eqref{Initial condition} comes from the fact that due to the non-local nature of the fractional differential-integral operators, the time-fractional pseudo-parabolic equation has not only provided both new insights into physical models but been also very mathematically interesting. In view of one of the the original ideas for the application of Equation \eqref{FluidEquation}, that is, the study of some fluid flows, it is natural to propose the fractional derivative in the investigation of some certain viscous fluids. In fact, many modified versions of \eqref{FluidEquation} have been proposed such as \cite{Bazhlekova,Fetecau}. Because the fractional derivative will help us to capture the viscoelastic properties of the flow, the time-fractional pseudo-parabolic equations are useful for describing the behavior of some non-Newtonian fluids. In mathematical aspect, it is a hot stream to consider the time-fractional version of the classical mathematical models including the parabolic type equations or diffusion models \cite{dif2,dif3,dif4,TomasTuan}, the time-space fractional Shr\"odinger equation with polynomial type nonlinearity \cite{fractional Schrodinger}, the time-fractional Navier-Stokes equations (FNS) \cite{fractional NavierStokes,Tomas1}, and also the time-fractional pseudo-parabolic equations \cite{Tuan pseudo}. Surprisingly in \cite{fractional NavierStokes}, it was shown that the order of time-fractional derivative influences the regularity not only in time variable but also in the spatial variable. In \cite{Tuan pseudo}, a first attempt to explore the influence of the degree of nonlinearity on the dynamic behavior of the solution was conducted by considering the logarithmic nonlinearity and globally Lipschitz nonlinearity. All of the above achievements in this direction pushed us to consider not only the influence of the degree of the nonlinearities but also the order of the time-fractional derivative on the behavior of the solution to the time-fractional pseudo-parabolic equations with four distinguished types of nonlinearities, in which the degree of nonlinearities increase gradually.
\subsection{Structure of the work}
To provide an overview of the present paper, we give the outline of the work, including some summaries of the mathematical contributions. 

\noindent$ \bullet $ Section \ref{Preliminaries} provides some basic knowledge about function spaces, special operators, the mild formula, and some linear estimates.

\noindent$ \bullet $ In Section \ref{GBBM Eq}, we consider the source term $ H $ as a gradient type. To prove the well-posedness results for $ H $ in the advective form $ H(u)=(\eta\cdot\nabla)u $, we apply the technique for singular integration developed in \cite{Atienza} to overcome the difficulties arising in finding the proper functional space and proving the convergence in such space for the constructed Picard sequence, without restriction on the time interval and the smallness assumptions on the initial data. Also, in this section, for  $H(u)= (\eta\cdot\nabla)u+\nabla\cdot F(u) $, where $ F $ is a vector field, we firstly prove the local-in-time existence, then this solution is extended to the one in some larger time interval. As a consequence, the mild solution is shown to be the global-in-time solution or finite time blow up solution.

\noindent$ \bullet $ Section \ref{Local Lipschitzian} states our investigation of Problem \eqref{Main Equation}-\eqref{Initial condition} with the polynomial nonlinearity and the exponential nonlinearity. Since the nonlinear estimates for $ H $ in these cases require much more strict conditions for the parameters and dimension $ d $, 
some smallness assumptions for $ u_0 $ are necessary for getting the global well-posedness. Also, for the power-type source term $ H(u)=|u|^{p-1}u,~p>2 $, the use of fractional Hilbert spaces and Sobolev embeddings is very beneficial. The nonlinearity of exponential type is even more difficult for us to control because of its rapid growth. Fortunately, by using the Orlicz space, we overcome this challenge and obtain the desired results.

\section{Preliminaries}\label{Preliminaries}
Entire this work, we always use the letter $I$ and the notation $\mathscr{T}$ to abbreviate, respectively, an interval of time and a time point in $(0,\infty)$.
\subsection{Basic materials}
We first establish some functional space concepts. Suppose that $X$ is a Banach space associated with the norm $\norm{\cdot}_{X}$. We use the notation $C(I \to X)$ to denote the space of all continuous functions $w:I \rightarrow X$. If $I$ is compact, $C(I \to X)$ is a Banach space with the norm
\begin{align*}
\norm{w}_{C(I\rightarrow X)}:=\sup_{t\in I}\bnorm{w(t)}_{X}<\infty.
\end{align*}
Suppose that $ \Xi:\R\to\R^+ $ is a Young function, i.e, a convex function which is even, continuous on $ [0,\infty) $, and satisfies
\begin{align*}
	\lim\limits_{z\to\infty}\frac{\Xi(z)}{z}=\infty\qquad\text{and}\qquad\lim\limits_{z\to0}\frac{\Xi(z)}{z}=0.
\end{align*}
Then, we define the Orlicz space $ \Lxi $ as the space of all measurable functions $ w(x) $ such that
\begin{align*}
\int_{\Omega}\Xi\left(\frac{\left|w(x)\right|}{\kappa}\right){\d} x<\infty\quad\text{for some } \kappa>0.
\end{align*} 
The space $ \Lxi $ is a Banach space with respect to the Luxemburg norm
\begin{align*}
\norm{w}_{\Lxi}:=\inf\left\{\kappa\in\R~\Big|~\kappa>0,\int_{\Omega}\Xi\left(\frac{\left|w(x)\right|}{\kappa}\right){\d} x\le1\right\}.
\end{align*}
\begin{remark}Note that, we can use the framework of Orlicz space to cover the definitions of some well-known Lebesgue spaces as below 
	
\begin{enumerate}[(1)]
\item Assume that $ \Xi(z)=z^p $ with $ 1< p<\infty $. Then, $ \Lxi $ is the usual Lebesgue space $ L^p(\Omega) $.
\item Let $ \Xi $ be a Young function whose value equals to 0 on $ [-1,1] $ and is not be bound outside $ [-1,1] $. Then, $ \Lxi $ is the usual Lebesgue space $ L^\infty(\Omega) $.
\end{enumerate}
\end{remark}
\begin{remark}
From now on, we always use the symbols $ \Lxi $ to indicate the Orlicz space with the Young function $ \Xi(z)=e^{z^2}-1. $
\end{remark}
For the purpose of deriving main estimates for Problem \eqref{Main Equation}-\eqref{Initial condition} involving the exponential nonlinearity, we introduce the following lemma which can be found in \cite[Theorem 8.12]{Sobolev spaces} or \cite[Lemma 2.1]{Ioku} about the relationship between the space $ \Lxi $ and some usual Lebesgue spaces. For readers' convenience, we only briefly present its proof.  
\begin{lemma}\label{Orlicz embeds into Lp}
For any $ p\in\left[2,\infty\right) $, we have the estimate 
\begin{align*}
\norm{w}_{L^p(\Omega)}\le\Big(\Gamma\left(\frac{p}{2}+1\right)\Big)^{\frac{1}{p}}\norm{w}_{\Lxi}.
\end{align*}
\end{lemma}
\begin{proof}
\noindent For $ q\ge1 $, using the properties of the exponential function and the Gamma function (see Definition \ref{Gamma-Beta defs}), we get
\begin{align*}
	\frac{z^q}{\Gamma(q+1)}+1<e^z.
\end{align*}
Then, from the fact that $ \Big\{\kappa\in\R~\big|~\kappa>0,\int_{\Omega}\Xi\big(|w(x)|\kappa^{-1}\d x\big)\le1\Big\}=\left[\norm{w}_{\Lxi},\infty\right) $ and the monotone convergence theorem, we obtain
\begin{align*}
	\int_{\Omega}\frac{\left(|w(x)|\norm{w}^{-1}_{\Lxi}\right)^{2q}}{\Gamma(q+1)}{\d} x\le\int_{\Omega}\Xi\left(\frac{|w(x)|}{\norm{w}_{\Lxi}}\right){\d} x\le1.
\end{align*}
Then, by choosing $ q=\frac{p}{2} $, we obtain the desired result.
\end{proof}
\noindent We note that the scalar product on $L^{2}(\Omega)$ between $w, v\in L^{2}(\Omega)$ is given by
\begin{align*} 
	\int_{\Omega}w( {z} )v( {z} )\mathrm{d} {z}.
\end{align*}
Then, from the spectral decomposition of the operator $\mathcal{A}$, for any $ \nu $, we can define the following Hilbert scale space
\begin{align*}
\Dnu:=\left\{w\in\L2~\Big|~\norm{w}^2_{\Dnu}=\sumuse\theta_k^\nu\left(\int_{\Omega}w( {z} )\phi_k(z) \mathrm{d} {z}\right)^2<\infty\right\}.
\end{align*}
In view of this setting, we define the space  $ D^{-\nu}(\Omega) $ by the dual space of $ \Dnu $ with respect to the pairing $ \left\langle\cdot,\cdot\right\rangle_* $,
which is a Banach space equipped with the norm
\begin{align*}
\norm{w}_{D^{-\nu}(\Omega)}=\left(\sumuse\theta_k^{-\nu}\left\langle w,\phi_k\right\rangle_*^2\right)^{\frac{1}{2}}.
\end{align*}
\begin{remark}
(\cite[Chapter 5]{Brezis}) If $ w\in\L2 $ and $ v\in\Dnu $, we have
\begin{align*}
\left\langle w,v\right\rangle_*=\left\langle w,v\right\rangle:=\int_{\Omega}w( {z} )v( {z} )\mathrm{d} {z}.
\end{align*}
\end{remark}

\begin{remark} Based on \cite[Section 3]{Vazquez}, we see that $ \Dnu $ coincides with the Sobolev-Slobodecki space $ W_0^{\nu,2}(\Omega) $  when $ \nu\in(1/2,1]. $
\end{remark}

We next introduce the definition of the Mittag-Leffler function with two parameters, which is the generalization of the classical Mittag-Leffler function, as follows
\begin{align*}
E_{\alpha,\zeta}( z):=\sum_{k=0}^{\infty}\frac{z^{k}}{\Gamma(\alpha k+\zeta)}, 
\end{align*}
where $\alpha$ is a positive real number and $\zeta$ is a complex constant. When $\zeta=1$, we use the symbol $E_{\alpha}$ instead of $E_{\alpha,1}$, and \cite[Section 1]{Podlubny} shows that
\begin{align*}
E_{\alpha}(-z):=\int_{0}^{\infty}e^{-z r}\mathcal{W}_{\alpha}(r)\mathrm{d}r,
\end{align*}
where $\mathcal{W}_{\alpha}$ is the M-Wright type function. Also, the following lemma is very useful to control the values of the Mittag-Leffler function. 

\begin{lemma}(\cite[Theorem 1.6]{Podlubny})\label{Bound of Mittag-Leffler}
Let $\alpha<1,~\zeta$ be a real constant, and $\lambda\in \left( \frac{\pi\alpha}{2}, \pi\right)$. Then, there exists a positive constant $\mathcal{M}$ such that
\begin{align*}
\Big|E_{\alpha,\zeta}(  \zeta)\Big|\leq\frac{\mathcal{M}}{1+|\zeta|}, 
\end{align*}
whenever $|\zeta|\geq 0$  and $\lambda\leq|\arg(\zeta)|\leq\pi$.
\end{lemma}
\begin{lemma}\label{Derivative of Mittag-Leffler}(\cite[Appendix E]{Mainardi})
Let $\alpha, t$ be in $(0,1)$ and $(0, \infty)$,  respectively. Then, for any $a>0$,  we have
\begin{align*}
\partial_{t}\big(E_{\alpha}(-at^{\alpha})\big)=-at^{\alpha-1}E_{\alpha,\alpha}(-at^{\alpha})\quad\text{and}\quad \partial_{t}\big(t^{\alpha-1}E_{\alpha,\alpha}(-at^{\alpha})\big)=t^{\alpha-2}E_{\alpha,\alpha-1}(-at^{\alpha}).
\end{align*}
\end{lemma}

\subsection{Mild solution}
For purpose of formulating a mild solution of Problem \eqref{Main Equation}-\eqref{Initial condition}, the following lemma on the Laplace transform of the Caputo derivative operator plays an important role.
\begin{lemma}
Suppose that the Laplace transform and the derivative of order $\alpha$  in the Caputo sense of a function $w$  exist. Then, the equation below holds
\begin{align*}
\widetilde{\partial_{t}^{\alpha}w}(s)=\mathcal{L}\{\partial_{t}^{\alpha}w\}(s)=s^{\alpha}\widetilde{w}(s)-s^{\alpha-1}w(0),
\end{align*}
where $ \mathcal{L} $ stands for the Laplace transform operator.
\end{lemma}
Integrating the both sides of Equation \eqref{Main Equation} with an arbitrary function $\phi_{k}$ and then applying the Laplace transform, we have
\begin{align*}
&s^{\alpha}  \int_{\Omega}[I+   \mathcal{A}]\widetilde{u}( {z} , s)\phi_{k}( {z} )\mathrm{d} {z} +\int_{\Omega}\mathcal{A}\widetilde{u}( {z} ,s)\phi_{k}( {z} )\mathrm{d} {z} \nonumber\\
=&s^{\alpha-1}  \int_{\Omega}[I+   \mathcal{A}]u_{0}( {z} )\phi_{k}( {z} )\mathrm{d} {z} +\int_{\Omega}\widetilde{H(u)}( {z} , s)\phi_{k}( {z} )\mathrm{d} {z} .
\end{align*}
This equation is equivalent to
\begin{align*}
\int_{\Omega}\widetilde{u}( {z} , s)\phi_{k}( {z} )\mathrm{d} {z} =&\frac{s^{\alpha-1}( 1+ \theta_{k})}{s^{\alpha}( 1+ \theta_{k})+\theta_{k}}\int_{\Omega}u_{0}( {z} )\phi_{k}( {z} )\mathrm{d} {z} \nonumber\\
&+  \frac{1}{s^{\alpha}( 1+ \theta_{k})+\theta_{k}}\int_{\Omega}\widetilde{H(u)}( {z} , s)\phi_{k}( {z} )\mathrm{d} {z} .
\end{align*}
Then by the inverse Laplace transform, we obtain
\begin{align*}
\int_{\Omega}u( {z} , t)\phi_{k}( {z} )\mathrm{d} {z} =&\int_{\Omega}E_{\alpha}\left(\frac{-\theta_{k}t^{\alpha}}{1+ \theta_{k}}\right)u_{0}( {z} )\phi_{k}( {z} )\mathrm{d} {z} \nonumber\\
&+\int_{\Omega}\frac{(t-\tau)^{\alpha-1}}{ 1+ \theta_{k}}E_{\alpha,\alpha}\left(\frac{-\theta_{k}(t-\tau)^{\alpha}}{ 1+ \theta_{k}}\right)H(u( {z} , \tau))\phi_{k}( {z} )\mathrm{d} {z}  \d\tau.
\end{align*}
Since $\{\phi_{k}\}_{k\in \mathbb{N}}$ is an orthonormal basis of $L^{2}(\Omega)$, we can formulate the mild solution to Problem \textbf{\eqref{Main Equation}-\eqref{Initial condition}} in the following way
\begin{align*}
u(x,t)&=\int_{\Omega}G_{1}(x,  {z} , t)u_{0}( {z} )\mathrm{d} {z} +\int_{0}^{t}\int_{\Omega}G_{2}(x,  {z} , t-\tau)H(u( {z} , \tau))\mathrm{d} {z}  \d\tau,
\end{align*}
where, the Green kernels $G_{1}, G_{2}$ are given by
\begin{flalign*}
\qquad G_{1}(x, {z} ,t)&=\sum_{k=1}^{\infty}E_{\alpha}\left(\frac{-\theta_{k}t^{\alpha}}{1+ \theta_{k}}\right)\phi_{k}(x)\phi_{k}( {z} )
\end{flalign*}
and
\begin{align*}
\qquad G_{2}(x, {z} ,t)&=\sum_{k=1}^{\infty}\frac{t^{\alpha-1}}{ 1+ \theta_{k}}E_{\alpha,\alpha}\left(\frac{-\theta_{k}t^{\alpha}}{ 1+ \theta_{k}}\right)\phi_{k}(x)\phi_{k}( {z} ).
\end{align*}
For purpose of simplifying notations, we set
\begin{align*}
\mathscr{S}(t)w(x):=\int_{\Omega}G_{1}(x,  {z} , t)w( {z} )\mathrm{d} {z}\quad\text{and}\quad
\mathscr{R}(t)w(x):=\int_{\Omega}G_{2}(x,  {z} , t)w(z)\mathrm{d} {z}.
\end{align*}
Then, we can rewrite the formula for the mild solution in the following way
\begin{align*}
u(x,t)=\mathscr{S}(t)u_0(x)+\int_{0}^{t}\mathscr{R}(t-\tau)H(u(x,\tau))\d\tau.
\end{align*}
From the standpoint of this formulation concept, we derive some fundamental linear estimate through the following lemma
\begin{lemma}\label{Linear estimate}
Let $ \alpha,\mu\in(0,1),~\nu\in[0,1]$ and $\nu^*\in[0,2]$. Assume that $ w\in\D1 $. Then, we can find positive constants $C_{1}$ and $C_{2}$ such that
\begin{enumerate}[(i)]
\item $\bnorm{\mathscr{S}(t)}_{\mathscr{L}\left(D^{\nu}(\Omega)\right)}\leq C_{1}t^{-\alpha\mu}$;
\item $\bnorm{\mathscr{R}(t)}_{\mathscr{L}\left(D^{{\nu-\nu^*}}(\Omega),\Dnu\right)}\leq C_{2}t^{\alpha-1}$.
\end{enumerate}
\end{lemma}
\begin{proof}~\\
(\textit{i}) We consider three different cases of $ \mu $ as follows to prove the conclusion \textit{(i)}

\noindent Case 1: $ \mu=0$. The following inequality follows immediately from Lemma \ref{Bound of Mittag-Leffler}
\begin{align*}
\bnorm{\mathscr{S}(t)w}_{D^{\nu}(\Omega)}^{2}&=\sum_{k=1}^{\infty}\theta^\nu_{k}\left(E_{\alpha}\left(\frac{-\theta_{k}t^{\alpha}}{ 1+\theta_{k}}\right)\right)^{2}\binner{w}{\phi_{k}}^{2}\\
&\le\sum_{k=1}^{\infty}\theta^\nu_{k}\left(\frac{\mathcal{M}}{1+\frac{\theta_{k}t^{\alpha}}{ 1+\theta_{k}}}\right)^2\binner{w}{\phi_{k}}^{2}\\
&\le\mathcal{M}^2\norm{w}_{\Dnu}.
\end{align*}

\noindent Case 2: $ \mu=1 $.
For a given $w$ in $ D^{\nu}(\Omega)$, Lemma \ref{Bound of Mittag-Leffler} and Parseval's identity show us
\begin{align*}
\bnorm{\mathscr{S}(t)w}_{D^{\nu}(\Omega)}^{2}&=\sum_{k=1}^{\infty}\theta^\nu_{k}\left(E_{\alpha}\left(\frac{-\theta_{k}t^{\alpha}}{ 1+\theta_{k}}\right)\right)^{2}\binner{w}{\phi_{k}}^{2}\nonumber\\
&\leq\sum_{k=1}^{\infty}\frac{2\mathcal{M}^{2}\theta^\nu_{k}(1+\theta_k^{2})}{\theta_{k}^{2}t^{2\alpha}}\binner{w}{\phi_{k}}^{2}\\
&\le2\mathcal{M}^2\left(1+\frac{1}{\theta^2_{1}}\right)t^{-2\alpha}\norm{u_0}_{\Dnu}.
\end{align*} 
\noindent Case 3: $ 0<\mu<1 $. Given $w\in D^{\nu}(\Omega)$, by analogous arguments as in above, we get
\begin{align}\label{Case 3 sp}
\bnorm{\mathscr{S}(t)w}_{D^{\nu}(\Omega)}^{2}=\sum_{k=1}^{\infty}\theta_{k}^{\nu}\left(E_{\alpha}\left(\frac{-\theta_{k}t^{\alpha}}{ 1+\theta_{k}}\right)\right)^{2}\binner{w}{\phi_{k}}^{2}.
\end{align}
Thanks to the relationship between the Mittag-Leffler function and the M-Wright type function, the equality \eqref{Case 3 sp} becomes
\begin{align*}
	\bnorm{\mathscr{S}(t)w}_{D^{\nu}(\Omega)}^{2}&=\sum_{k=1}^{\infty}\theta_{k}^{\nu}\left(\int_{0}^{\infty}\exp\left(\frac{-r\theta_{k}t^{\alpha}}{ 1+\theta_{k}}\right)\mathcal{W}_{\alpha}(r)\mathrm{d}r\right)^{2}\binner{w}{\phi_{k}}^{2}\nonumber\\
	&\leq \sum_{k=1}^{\infty}C_{\mu}^{2}t^{-2\alpha\mu}\left(\frac{\theta_{k}^{\nu}( 1+\theta_{k})^{2\mu}}{\theta_{k}^{2\mu}}\right)\left(\int_{0}^{\infty}r^{-\mu}\mathcal{W}_{\alpha}(r)\mathrm{d}r\right)^{2}\binner{w}{\phi_{k}}^{2},
\end{align*}
wherein, we have used the fundamental inequality $e^{-z}\leq C_{\mu}z^{-\mu}, \mu\in(0,1), C_{\mu}>0$. Then, Lemma \ref{Property of M-Wright function} implies
\begin{align*}
	\bnorm{\mathscr{S}(t)w}_{D^{\nu}(\Omega)}\leq C_{\mu}t^{-\alpha\mu}(\theta_{1}^{-1}+1)^{\mu}\frac{\Gamma(1-\mu)}{\Gamma(1-\alpha\mu)}\norm{w}_{D^{\nu}(\Omega)}.
\end{align*}

\noindent({\it ii}) We consider only the case $ \nu^*=1 $, the other cases are similar. For given $w\in \Dnuminus$, we observe that
\begin{align*}
\bnorm{\mathscr{R}(t)w}_{D^{\nu}(\Omega)}^{2}\leq\sum_{k=1}^{\infty}\frac{\theta^\nu_{k}t^{2\alpha-2}}{(1+\theta_{k})^{2}}\left(E_{\alpha,\alpha}\left(\frac{-\theta_{k}t^{\alpha}}{ 1+\theta_{k}}\right)\right)^{2}\binner{w}{\phi_{k}}^{2}.
\end{align*}
From this, we can easily obtain the positive constant $C_{2}$ via the inequality below
\begin{align*}
\bnorm{\mathscr{R}(t)w}_{D^{\nu}(\Omega)}\leq\frac{\mathcal{M}t^{\alpha-1}}{\theta_1^{\frac{1}{2}}}\left(\sum_{k=1}^{\infty}\theta_k^{\nu-1}\binner{w}{\phi_{k}}^{2}\right)^{\frac{1}{2}},
\end{align*}
noting that we have applied Lemma \ref{Bound of Mittag-Leffler} to get
\begin{align*}
E_{\alpha,\alpha}\left(\frac{-\theta_{k}t^{\alpha}}{ 1+\theta_{k}}\right)\le\frac{\mathcal{M}}{1+\left|\frac{-\theta_{k}t^{\alpha}}{ 1+\theta_{k}}\right|}.
\end{align*}
The proof is completed.
\end{proof}
\section{The generalized BBM equation}\label{GBBM Eq}
In this section, we investigate the initial value problem involving the generalized BBM equation. More precisely, Equation \eqref{Main Equation} can be given in the exact form
\begin{align}\label{GBBM Equation (for eqref)}
	\partial^\alpha_{t}u-\partial^\alpha_{t}\Delta u-b\Delta u+(\eta\cdot\nabla)u+\nabla\cdot F(u)=0,
\end{align}
whereby, $ \eta $ is a $ d-$dimensional constant vector and $ F $ is a $ d-$dimensional vector field.
\subsection{The time-fractional pseudo-parabolic equation with advection term}
Throughout the current subsection, we assume that $ F $ is the vector $ 0 $ in $ H(u)=(\eta\cdot\nabla)u+\nabla\cdot F(u) $. Then, $ H $ becomes a globally Lipschitz source function. Indeed, by Cauchy–Schwarz inequality, we have
\begin{align*}
\int_{\Omega}\left(\sum_{j=1}^d\eta_ju_{x_j}(x)\right)^2{\d} x\le|\eta|^2\int_{\Omega}\left(\sum_{j=1}^du^2_{x_j}(x)\right){\d} x=|\eta|^2\int_{\Omega}|\nabla u(x)|^2{\d} x.
\end{align*} 
It means that we can find a positive constant $\mathscr{C}_1$ independent of $ w,v\in\D1 $ such that
\begin{align}\label{Global Lipschitzian}
\bnorm{H(w)-H(v)}_{L^{2}(\Omega)}\leq \mathscr{C}_{1}\norm{w-v}_{D^{1}(\Omega)}.
\end{align}
Our main principle is the successive approximation in some reasonable Banach spaces. To this end, for any $ a,\sigma>0 $ we denote by $\mathbb{Y}=\mathbb{Y}(a,\sigma)$, the space of all functions $w\in C\left(I \to D^{1}(\Omega)\right)$ satisfying
\begin{flalign*}
	\qquad&\bullet~\sup_{t\in I\setminus \{0\}}t^{\kappa}e^{-\sigma t}\norm{w(t)}_{D^{1}(\Omega)}<\infty,&\\
	&\bullet~w(x, 0)=u_{0}(x),
\end{flalign*}
and construct a sequence $\{w_{n}\}_{n=1}^{\infty}$ inductively in the following way
\begin{align*}
	w_{1}(x,t)&:=\mathscr{S}(t)u_0(x) ,\\
	w_{n+1}(x,t)&:=\mathscr{S}(t)u_0(x)+\int_{0}^{t}\mathscr{R}(t-\tau)H(w_{n}(x,\tau))\d\tau.
\end{align*}
Besides, we introduce some necessary lemmas to help us present the main points of the proof more clearly.
\begin{lemma}\label{Weighted_Limit}(\cite[Lemma 8]{Atienza})
	Let $h,\mu>0$  and $m, n>-1$  such that $m+n>-1$. Then,
	\begin{align*}
		\lim_{\mu\rightarrow\infty}\left(\sup_{t\in I\setminus \{0\}}t^{h}\int_{0}^{1}s^{m}(1-s)^{n}e^{-\mu t(1-s)}\mathrm{d}s\right)=0.
	\end{align*}
\end{lemma}
\begin{lemma}\label{Sect3_continuous}
	Presume that $u_{0}$  belongs to $D^{1}(\Omega)$  and $I =[0, \mathscr{T}]$. Then, $\{w_{n}\}_{n=1}^{\infty}$  forms a subset of $C \left(I \rightarrow D^{1}(\Omega)\right)$.
\end{lemma} 
\begin{proof}
	Firstly, for any $t\geq 0$ and $\varepsilon>0$, Lemma \ref{Derivative of Mittag-Leffler} makes the following formula hold
	\begin{align*}
		\left|E_{\alpha}\left(  \frac{-\theta_{k}(t+\varepsilon )^{\alpha}}{ 1+\theta_{k}}\right)-E_{\alpha}\left(\frac{-\theta_{k}t^{\alpha}}{ 1+\theta_{k}}\right)\right|=\int_{t}^{t+\epsilon}\frac{\theta_{k}r^{\alpha-1}}{ 1+\theta_{k}}E_{\alpha,\alpha}\left(\frac{-\theta_{k}r^{\alpha}}{ 1+\theta_{k}}\right) \mathrm{d}r.
	\end{align*}
	Then, for $u_{0}\in D^{1}(\Omega)$, one deduces that
	\begin{align*}
		\bnorm{w_{1}(t+\varepsilon )-w_{1}(t)}_{D^{1}(\Omega)}^{2}&=\sum_{k=1}^{\infty}\theta_{k}\left|E_{\alpha}\left(\frac{-\theta_{k}(t+\varepsilon )^{\alpha}}{ 1+\theta_{k}}\right)-E_{\alpha}\left(\frac{-\theta_{k}t^{\alpha}}{ 1+\theta_{k}}\right)\right|^{2}\big\langle u_0,\phi_{k}\big\rangle^{2}\nonumber\\
		&\le\sumuse\theta_{k}\left(\int_{t}^{t+\epsilon}\frac{r^{\alpha-1}\mathcal{M}}{1+\left|\frac{-\theta_{k}r^{\alpha}}{ 1+\theta_{k}}\right|} \mathrm{d}r\right)^2\big\langle u_0,\phi_{k}\big\rangle^{2}\\
		&\leq\frac{\mathcal{M}^2\big((t-\varepsilon)^{\alpha}-t^\alpha\big)^2}{\alpha^{2}}\norm{u_0}^2_{\D1}.
	\end{align*}
One may obtain immediately for any $t\geq 0$ that
\begin{align}\label{Continuity w_1}
\lim_{\varepsilon\to 0}\norm{w_{1}(t+\varepsilon )-w_{1}(t)}_{D^{1}(\Omega)}=0,
\end{align}
which means that $w_{1}$ is continuous on $I$ with respect to the $D^{1}(\Omega)$ norm. On the other hand, for every $\varepsilon>0$, we get
	\begin{align*}
		&\bnorm{ w_{2}(\varepsilon)-w_{1}(\varepsilon)}_{D^{1}(\Omega)}\\
		\leq&\int_{0}^{\varepsilon}\Bnorm{\mathscr{R}(\varepsilon-\tau)H(w_{n}(x,\tau))}_{D^{1}(\Omega)}\d\tau\nonumber\\
		=&\int_{0}^{\varepsilon}\Bigg(\sum_{k=1}^{\infty}\frac{\theta_{k}(\varepsilon -\tau)^{2(\alpha-1)}}{( 1+\theta_{k})^{2}}E^2_{\alpha,\alpha}\left(\frac{\theta_{k}(\tau-\varepsilon)^{\alpha}}{ 1+\theta_{k}}\right)\Binner{H(w_1(\tau))}{\phi_{k}}^{2}\Bigg)^{\frac{1}{2}}\d\tau.
	\end{align*}
	Using Lemma \ref{Linear estimate}({\it ii}) and the fact that $ H(0)=0 $, we deduce
	\begin{align*}
		\bnorm{ w_{2}(\varepsilon)-w_{1}(\varepsilon)}_{D^{1}(\Omega)}&\leq C_{2}\int_{0}^{\varepsilon}(\varepsilon -\tau)^{\alpha-1}\Bnorm{H(w_{1}(\tau))}_{L^{2}(\Omega)}\d\tau\nonumber\\
		&\leq \mathscr{C}_{1}C_{2}\left(\int_{0}^{\varepsilon}(\varepsilon -\tau)^{\alpha-1}\d\tau\right)\left(\sup_{t\in I}\bnorm{w_{1}(t)}_{D^{1}(\Omega)}\right).
	\end{align*}
	This result along with \eqref{Continuity w_1} help us conclude that $w_{2}$ is continuous at $t=0$ with respect to the $D^{1}(\Omega)$ norm. We now turn to consider the case when $t$ is positive. For this purpose, we proceed with two following claims, which help us make the proof more clear.
	
	\noindent \textbf{Claim 1}. For any $t, \varepsilon>0, k\in \mathbb{N}$, and $w_{1}\in C \left(I \rightarrow D^{1}(\Omega)\right)$, we can apply Lemma \ref{Derivative of Mittag-Leffler} as follows
	\begin{align*}
		\mathscr{E}(k, t,\tau, \varepsilon):&=\left|(t+\varepsilon -\tau)^{\alpha-1}E_{\alpha,\alpha}\left(\frac{-\theta_{k}(t+\varepsilon -\tau)^{\alpha}}{1+\theta_k}\right)-(t -\tau)^{\alpha-1}E_{\alpha,\alpha}\left(\frac{-\theta_{k}(t-\tau)^{\alpha}}{ 1+\theta_{k}}\right)\right|\nonumber\\
		&=\left|\int_{t-\tau}^{t+\varepsilon-\tau}\frac{\theta_kr^{\alpha-2}}{1+\theta_{k}}E_{\alpha,\alpha-1}\left(\frac{-\theta_{k}r^{\alpha}}{ 1+\theta_{k}}\right)\mathrm{d}r\right|.
	\end{align*}
	Lemma \ref{Bound of Mittag-Leffler} implies that, for any $\tau\in(0, t)$,
	\begin{align}\label{Inequality for E(k,t,tau,epsilon)}
		\mathscr{E}(k, t, \tau, \varepsilon)\leq\mathcal{M}(1-\alpha)^{-1}\Big|(t+\varepsilon -\tau)^{\alpha-1}-(t-\tau)^{\alpha-1}\Big|.
	\end{align}
	In view of this result, we infer the following estimate
	\begin{align*}
		&\int_{0}^{t}\Bigg(\sum_{k=1}^{\infty}\frac{\theta_{k}\mathscr{E}^2(k,t,\tau,\varepsilon)}{( 1+\theta_{k})^{2}}\Binner{H(w_1(\tau))}{\phi_{k}}^{2}\Bigg)^{\frac{1}{2}}\d\tau\nonumber\\
		\leq&\frac{\mathcal{M}\mathscr{C}_{1}}{(1-\alpha)\theta_1^{\frac{1}{2}}}\left(\int_{0}^{t}\Big|(t+\varepsilon  -\tau)^{\alpha-1}-(t-\tau)^{\alpha-1}\Big|\d\tau\right)\left(\sup_{t\in I}\norm{w_{1}(t)}_{D^{1}(\Omega)}\right).
	\end{align*}
	We notice that, for any $\tau\in(0, t)$, we have
	\begin{align*}
		\begin{cases}
			\dis\lim_{\varepsilon\rightarrow\infty}\Big((t-\tau)^{\alpha-1}-(t+\varepsilon  -\tau)^{\alpha-1}\Big)=0,\\[0.2cm]
			\dis\Big|(t+\varepsilon  -\tau)^{\alpha-1}-(t-\tau)^{\alpha-1}\Big|\leq 2(t-\tau)^{\alpha-1}.
		\end{cases}
	\end{align*}
	Then, applying the dominated convergence theorem, the following limit holds
	\begin{align*}
		\lim_{\varepsilon\to 0}\left(\int_{0}^{t}\Bigg(\sum_{k=1}^{\infty}\frac{\theta_{k}\mathscr{E}^2(k,t,\tau,\varepsilon)}{( 1+\theta_{k})^{2}}\Binner{H(w_1(\tau))}{\phi_{k}}^{2}\Bigg)^{\frac{1}{2}}\d\tau\right)=0.
	\end{align*}
	
	\noindent \textbf{Claim 2}. Using again Lemma \ref{Linear estimate}({\it ii}) and the fact that $w_{1}\in C\left(I \to D^{1}(\Omega)\right)$, we obtain
	\begin{align*}
		&\int_{t}^{t+\epsilon}\Bnorm{\mathscr{R}(t+\varepsilon-\tau)H(w_{n}(x,\tau))}_{D^{1}(\Omega)}\mathrm{d}\tau\\
		\leq& C_{2}\int_{t}^{t+\epsilon}(t+\varepsilon  -\tau)^{\alpha-1} \Bnorm{H(w_1(\tau))}_{\L2}\d\tau\nonumber\\
		\leq& \mathscr{C}_{1}C_{2}\alpha^{-1}\varepsilon^{\alpha}\left(\sup_{t\in I}\norm{w_{1}(t)}_{D^{1}(\Omega)}\right).
	\end{align*}
	Thus, one has
	\begin{align*}
		\lim_{\varepsilon\to 0}\left(  \int_{t}^{t+\epsilon}\Bnorm{\mathscr{R}(t+\varepsilon-\tau)H(w_{n}(x,\tau))}_{D^{1}(\Omega)}\d\tau\right)=0.
	\end{align*}
	
	Now, the results of {\it Claim 1} and {\it Claim 2} along with the fact that $w_{1}\in C \left(I \rightarrow D^{1}(\Omega)\right)$ yield that $w_{2}$ is continuous at $t>0$. Note that, we have already proved the continuity at $0$ of $w_{2}$. Therefore, we can conclude that $w_2\in C \left(I \rightarrow D^{1}(\Omega)\right)$. Based on this result, the induction arguments show that $w_{n+1}$ belongs to $C \left(I \rightarrow D^{1}(\Omega)\right)$ whenever $w_{n}$ is in $C\left(I \to D^{1}(\Omega)\right)$, for every $n\in \mathbb{N}$. The proof is completed. 
\end{proof}
\begin{theorem} (Global existence)
	Assume that $u_{0}$ is an element of the space $D^{1}(\Omega)$  and $I =[0, \mathscr{T}]$. Then, Problem $\eqref{Main Equation}$-$\eqref{Initial condition}$ possesses at least one mild solution $u$  in $C \left(I \rightarrow D^{1}(\Omega)\right)$.
\end{theorem}
\begin{proof}
We first show that if $u_{0}$  belongs to $D^{1}(\Omega)$  and $I =[0, \mathscr{T}]$, $\{w_{n}\}_{n=1}^{\infty}$ is a subset of the space $\mathbb{Y}(\alpha, \sigma)$,  for some big-value $\sigma$. If $u_{0}\in D^{1}(\Omega)$, Lemma \ref{Linear estimate}({\it i}) shows that
\begin{align*}
	\norm{w_{1}(t)}_{D^{1}(\Omega)}=   \bnorm{\mathscr{S}(t)u_0}_{D^{1}(\Omega)}\leq C_{1}t^{-\alpha}\norm{u_{0}}_{D^{1}(\Omega)}.
\end{align*}
It thus follows that
\begin{align}\label{Y norm of w1}
	\sup_{t\in I\setminus \{0\}}t^{\alpha}e^{-\sigma t}\norm{w_{1}(t)}_{D^{1}(\Omega)}<\infty.
\end{align}
Next, we suppose that $w_{n}\in \mathbb{Y}(\alpha, \sigma)$. Then, by using Lemma \ref{Linear estimate}({\it ii}), we have
\begin{align*}
	&\norm{w_{n+1}(t)-w_{1}(t)}_{D^{1}(\Omega)}\\
	=&\Bgnorm{\int_{0}^{t}\mathscr{R}(t-\tau)H(w_n(\tau))\d\tau}\\
	\leq& C_{2}\int_{0}^{t}(t-\tau)^{\alpha-1}\bnorm{ H(w_{n}(\tau))}_{L^{2}(\Omega)}\d\tau\nonumber\\
	\leq& \mathscr{C}_{1}C_{2}\int_{0}^{t}(t-\tau)^{\alpha-1}\norm{w_{n}(\tau)}_{D^{1}(\Omega)}\d\tau.
\end{align*}
Multiplying the both sides of the above by $t^{\alpha}e^{-\sigma t}, t>0$, we obtain
\begin{align}\label{Y norm of w2}
	t^{\alpha}e^{-\sigma t}  \norm{w_{n+1}(t)-w_{1}(t)}_{D^{1}(\Omega)}\leq \mathscr{C}_{1}C_{2}\mathcal{Q}\left(t,\alpha,\sigma\right)\left(\sup_{t\in I}t^{\alpha}e^{-\sigma t}\Vert w_{n}(t)\Vert_{D^{1}(\Omega)}\right),
\end{align}
where we set
\begin{align*}
	\mathcal{Q}\left(t,h,\sigma\right):=t^{h}\int_{0}^{t}(t-\tau)^{\alpha-1}\tau^{-h}e^{-\sigma(t-\tau)}\d\tau.
\end{align*}
By virtue of Lemma \ref{Weighted_Limit},  \eqref{Y norm of w1}, \eqref{Y norm of w2} and the triangle inequality, we deduce
\begin{align*}
\sup_{t\in I\setminus \{0\}}\Big(t^{\alpha}e^{-\sigma t}\norm{w_{n+1}(t)}_{D^{1}(\Omega)}\Big)<\infty\quad\text{for some large } \sigma.
\end{align*}
In addition, the use of Lemma \ref{Sect3_continuous} yields that $w_{n}$ is continuous on $I$, for every $n\geq 1$. It follows that $w_{n+1}\in \mathbb{Y}(\alpha, \sigma)$. Furthermore, $\{w_{n}\}_{n=1}^{\infty}$ can be proved to be a Cauchy sequence in $\mathbb{Y}(\alpha, \sigma_{0})$. Let us verify this statement.  Firstly, Lemma \ref{Weighted_Limit} ensures the existence of a sufficiently large constant $\sigma_{0}$ such that
\begin{align*}
	\sup_{t\in I\setminus \{0\}}
	\mathcal{Q}\left(t,\alpha,\sigma_{0}\right)=\frac{3}{4\mathscr{C}_{1}C_{2}}.
\end{align*}
Secondly, assume that $w_{n-1}, w_{n}\in \mathbb{Y}(\alpha, \sigma_{0}), n\geq 2$. Then, by the same arguments as above, we obtain
\begin{align*}
&\norm{w_{n+1}(t)-w_{1}(t)}_{D^{1}(\Omega)}\\
\leq& C_{2}\int_{0}^{t}(t-\tau)^{\alpha-1}\Vert H(w_{n}(\tau))-H(w_{n-1}(\tau))\Vert_{L^{2}(\Omega)}\d\tau\nonumber\\
\leq& \mathscr{C}_{1}C_{2}\int_{0}^{t}(t-\tau)^{\alpha-1}\Vert w_{n}(\tau)-w_{n-1}(\tau)\Vert_{D^{1}(\Omega)}\d\tau.
\end{align*}
It thus implies that
\begin{align*}
	t^{\alpha}e^{-\sigma_{0}t}\norm{w_{n+1}(t)-w_{n}(t)}_{D^{1}(\Omega)}\leq\frac{3}{4}\sup_{t\in I\setminus \{0\}}\Big(t^{\alpha}e^{-\sigma_{0}t}\norm{w_{n}(t)-w_{n-1}(t)}_{D^{1}(\Omega)}\Big).
\end{align*}
From the standpoint of this result, for any $n_{2}>n_{1}\geq 2$, we find
\begin{align}\label{Cauchy-Picard argument}
	t^{\alpha}e^{-\sigma_{0}t}  \norm{ w_{n_{2}}(t)-w_{n_{1}}(t)}_{D^{1}(\Omega)}\leq\sum_{n=n_{1}}^{n_{2}}\left(\frac{3}{4}\right)^{n-1}\sup_{t\in I\setminus \{0\}}\Big(t^{\alpha}e^{-\sigma_{0}t}\Vert w_{2}(t)-w_{1}(t)\Vert_{D^{1}(\Omega)}\Big).
\end{align} 
Then, the use of geometric series helps us to declare that $\{w_{n}\}_{n=1}^{\infty}$ is a Cauchy sequence in $\mathbb{Y}(\alpha, \sigma_{0})$. Now, the completeness of $\mathbb{Y}(\alpha, \sigma_{0})$ provides a function $u$ which is the limit of the sequence $\{w_{n}\}_{n=1}^{\infty}$. Applying the dominated Lebesgue theorem, we have
\begin{align*}
	u&=\lim_{n\rightarrow\infty}\Bigg(\int_{\Omega}G_{1}(x,  {z} , t)u_{0}( {z} )\d {z} +\int_{0}^{t}\int_{\Omega}G_{2}(x,  {z} , t-\tau)H(w_{n}( {z} , \tau))\d {z}  \d\tau\Bigg)\nonumber\\
	&=\int_{\Omega}G_{1}(x,  {z} , t)u_{0}( {z} )\d {z} +\int_{0}^{t}\int_{\Omega}G_{2}(x,  {z} , t-\tau)H(u( {z} , \tau))\d {z}  \d\tau.
\end{align*}
Then,  we can conclude that $u$ is the mild solution to Problem \eqref{Main Equation}-\eqref{Initial condition}.
\end{proof}
\begin{theorem}(Uniqueness and stability)
	Suppose that Problem $\eqref{Main Equation}$-$\eqref{Initial condition}$ admits a mild solution $u$  in $C\left(I \to D^{1}(\Omega)\right)$. Then, this solution is globally unique. Furthermore, it depends continuously on the initial data.
\end{theorem}
\begin{proof}
The uniqueness and the stability results of the mild solution can be obtained by G\"onwall's inequality. Indeed, assume that $ u,w $ are the milds solution of Problem \eqref{Main Equation}-\eqref{Initial condition} in $ \mathbb{Y}(\alpha,\sigma_{0}) $ corresponding to the initial data $ u_0,w_0\in\D1 $. Then, we have
\begin{align}\label{Sect3_Stability_sp1}
&\norm{u(t)-w(t)}_{D^{1}(\Omega)}\nonumber\\
\leq& C_{1}t^{-\alpha}\norm{u_{0}-w_0}_{D^{1}(\Omega)}
\nonumber\\
&+
\mathscr{C}_{1}C_{2}\int_{0}^{t}(t-\tau)^{\alpha-1}\Vert u(\tau)-w(\tau)\Vert_{D^{1}(\Omega)}\d\tau.
\end{align}
Since $ u,w $ are in the space $ \mathbb{Y}(\alpha,\sigma_{0}) $, then, H\"older's inequality shows that
\begin{align*}
	&\left(t^\alpha e^{-\sigma_{0}t}\int_{0}^{t}(t-\tau)^{\alpha-1}\Vert u(\tau)-w(\tau)\Vert_{D^{1}(\Omega)}\d\tau\right)^2\nonumber\\
	\le&\mathcal{Q}\left(t,2\alpha,2\sigma_{0}\right)\int_{0}^{t}(t-\tau)^{\alpha-1}\left(\tau^\alpha e^{-\sigma_{0}\tau}\norm{u(\tau)-v(\tau)}_{\D1}\right)^2\d\tau.
\end{align*} 
Accordingly, the inequality \eqref{Sect3_Stability_sp1} becomes
\begin{align*}
&\left(t^{\alpha}e^{-\sigma_{0}t}\norm{u(t)-w(t)}_{D^{1}(\Omega)}\right)^2\\
\le&\left(C_1\norm{u_{0}-w_0}_{D^{1}(\Omega)}\right)^2+\frac{3}{4}\int_{0}^{t}(t-\tau)^{\alpha-1}\left(\tau^\alpha e^{-\sigma_{0}\tau}\norm{u(\tau)-v(\tau)}_{\D1}\right)^2\d\tau
\end{align*} 
We now apply the fractional Gr\"onwall's inequality (see Lemm \ref{Fractional Gronwall inequality}) to get the following estimate
\begin{align*}
	t^{\alpha}e^{-\sigma_{0}t}\norm{u(t)-w(t)}_{D^{1}(\Omega)}\le C_1\norm{u_{0}-w_0}_{D^{1}(\Omega)}E_{\alpha,1}\left(\frac{3\Gamma(\alpha)t^{\alpha}}{4}\right).
\end{align*}
The proof is completed.
\end{proof}
\subsection{The non-trivial case of the vector field}
In this subsection, we assume that the vector field $ F $ in $ H(u)=(\eta\cdot\nabla)u+\nabla\cdot F(u) $, is given by $ F(u)=(u^2,u^2,...,u^2) $ and $ d\in\{2,3,4\} $. Then, we can check that $ H $ is a local Lipschitzian. To verify this statement, we first recall the embeddings $ L^{\frac{2d}{d+2}}(\Omega)\hookrightarrow D^{-1}(\Omega) $ in Lemma \ref{Fractional Sobolev embeddings}, which helps us to find
\begin{align*}
\Bnorm{\nabla\cdot F(u)}_{D^{-1}(\Omega)}\le2^{\frac{2d+2}{d+2}}\Bnorm{\big|\nabla u\big|u}_{L^{\frac{2d}{d+2}}(\Omega)}\le2^{\frac{2d+2}{d+2}}\bnorm{u}_{L^d(\Omega)}\bnorm{\nabla u}_{\L2},
\end{align*}
	where the H\"older inequality with two conjugate indices $\frac{d+2}{2}$ and $ \frac{d+2}{d} $ has been applied to get the second inequality. We now can easily check that
\begin{align*}
\begin{cases}
\D1\hookrightarrow L^d(\Omega),\qquad&\text{if }~ d=2,\\
\D1\hookrightarrow L^{\frac{2d}{d-2}}(\Omega)\hookrightarrow L^d(\Omega),&\text{if }~d\in\{3,4\}.
\end{cases}
\end{align*} 
Therefore, there exists a positive constant $ C_{\Omega} $ from the above embeddings such that that
	\begin{align*}
		\Bnorm{\big|\nabla u\big|u}_{D^{-1}(\Omega)}\le 2^{\frac{2d+2}{d+2}}C_{\Omega}\bnorm{u}^2_{\D1}.
	\end{align*}
	As a consequence of the above estimates, we can find a constant $ C(\Omega,d) $ such that for any $ u,v\in\D1 $ there holds
	\begin{align*}
		\Bnorm{\nabla\cdot F(u)-\nabla\cdot F(v)}_{D^{-1}(\Omega)}\le C(\Omega,d)\left(\bnorm{ u}_{\D1}+\bnorm{ v}_{\D1}\right)\norm{u-v}_{\D1}.
	\end{align*}
In addition, from \eqref{Global Lipschitzian} and the fact that $ \L2\hookrightarrow D^{-1}(\Omega) $, for any $ u,v\in\D1 $, there exists a positive constant $ \overline{\mathscr{C}}_1 $ independent of $ u,v $ such that the following estimate holds
\begin{align}\label{Nonlinear estimate for Burgers source}
\bnorm{H\left(u\right)-H\left(v\right)}_{D^{-1}(\Omega)}\le\overline{\mathscr{C}}_1\left(\bnorm{ u}_{\D1}+\bnorm{ v}_{\D1}+1\right)\norm{u-v}_{\D1}
\end{align}
The main aim of this part is to investigate the well-posed results that either the mild solution exists globally or blows up in finite maximal time.
\begin{theorem}(Local existence)
Let $ \alpha\in(0,1) $. If $ u_0\in\D1 $, Problem \eqref{Main Equation}-\eqref{Initial condition} posesses a locally unique mild solution. 
\end{theorem}
\begin{proof}
	Thanks to the assumption $ u_0\in\D1 $, Lemma \ref{Linear estimate}\textit{(i)} implies
	\begin{align}\label{Sp1 Sect4.2}
		\bnorm{\mathscr{S}(t)u_0}_{\D1}\le C_1\norm{u_0}_{\D1}.
	\end{align}
	For any $ w,v\in C(I\to\D1) $, from Lemma \ref{Linear estimate}\textit{(ii)} and inequality \eqref{Nonlinear estimate for Burgers source} we have
	\begin{align*}
		&\int_{0}^{t}\bnorm{\mathscr{R}(t-\tau)\Big(H(w(\tau))-H(v(\tau))\Big)}_{\D1}\d\tau\nonumber\\
		\le& C_2\int_{0}^{t}(t-\tau)^{\alpha-1}\norm{H(w(\tau))-H(v(\tau))}_{D^{-1}(\Omega)}\d\tau\nonumber\\
		\le& \overline{\mathscr{C}}_1C_2\int_{0}^{t}(t-\tau)^{\alpha-1}\left(\norm{ w(\tau)}_{\D1}+\norm{ v(\tau)}_{\D1}+1\right)\norm{w(\tau)-v(\tau)}_{\D1}\d\tau.
	\end{align*}
This implies that
\begin{align}\label{Sp2 Sect4.2}
&\int_{0}^{t}\bnorm{\mathscr{R}(t-\tau)\Big(H(w(\tau))-H(v(\tau))\Big)}_{\D1}\d\tau
\nonumber\\
\le&\overline{\mathscr{C}}_1C_2\alpha^{-1}\mathscr{T}^\alpha\Bigg(\norm{w}_{C(I\to\D1)}+\norm{v}_{C(I\to\D1)}+1\Bigg)\norm{w-v}_{C(I\to\D1)}.
\end{align}
	Then, if we consider the Picard sequence
	\begin{flalign*}
		w_{1}(x,t)&:=\mathscr{S}(t)u_0(x),\\
		w_{n+1}(x,t)&:=\mathscr{S}(t)u_0(x)+\int_{0}^{t}\mathscr{R}(t-\tau)H(w_{n}(x,\tau))\d\tau,
	\end{flalign*}
	Inequality \eqref{Sp1 Sect4.2} immediately shows that $ w_1\in C(I\to\D1) $. Besides, for $ w_n\in C(I\to\D1) $, Estimate \eqref{Sp2 Sect4.2} gives us the following result
	\begin{align*}
		&\norm{w_{n+1}(\tau)-w_1(\tau)}_{\D1}\\
		\le&\int_{0}^{t}\bnorm{\mathscr{R}(t-\tau)H(w_n(\tau))}_{\D1}\d\tau\nonumber\\
		\le&\overline{\mathscr{C}}_1C_2\alpha^{-1}\mathscr{T}^\alpha\Big(\norm{w_n}^2_{C(I\to\D1)}+\norm{w_n}_{C(I\to\D1)}\Big),
	\end{align*} 
	i.e.,
	\begin{align*}
		\norm{w_{n+1}-w_1}_{C(I\to\D1)}\le\overline{\mathscr{C}}_1C_2\alpha^{-1}\mathscr{T}^\alpha\Big(\norm{w_n}^2_{C(I\to\D1)}+\norm{w_n}_{C(I\to\D1)}\Big).
	\end{align*}
	Then, we can conclude that $ \{w_n\}_{n=1}^\infty\subset C(I\to\D1). $ Furthermore, if $ \mathscr{T} $ is small enough, Estimate \eqref{Sp2 Sect4.2} infers for $ w_{n-1},w_n,w_{n+1}\in C(I\to\D1)$ that
	\begin{align*}
		\norm{w_{n+1}-w_{n}}_{C(I\to\D1)}\le\frac{3}{4}\norm{w_{n}-w_{n-1}}_{C(I\to\D1)}.
	\end{align*}
	It means $ \{w_n\}_{n=1}^\infty $ is a Cauchy sequence in $ C(I\to\D1) $. Consequently, there exists a locally unique mild solution $ u $ of Problem \eqref{Main Equation}-\eqref{Initial condition} in $ C(I\to\D1) $.
\end{proof}
\begin{lemma}\label{ExtensionLemmaSect4}
	Let $ u_0 $ be in $ \D1 $. Suppose that $ u $ is a unique mild solution to Problem \eqref{Main Equation}-\eqref{Initial condition} on $ [0,\mathscr{T}] $. Then this solution can be extended to $ [0,\mathscr{T}+T], $ for some $ T>0. $
\end{lemma}
\begin{proof}
	Assume that $ u $ is the unique mild solution of Problem \eqref{Main Equation}-\eqref{Initial condition} on $ [0,\mathscr{T}] $ and $ I=[0,\mathscr{T}+T] $, where $ T>0 $ will be chosen later. For a constant $  R>0 $, we consider the space below
	\begin{align*}
		\mathbb{E}:=\left\{w\in C(I\to  \D1)\quad\Biggl|
		\begin{array}{*{11}{ll}}
			&w(t,\cdot) =u(t,\cdot),&\forall t\in \left[0,\mathscr{T}\right],\\
			&\norm{w(t,\cdot)-u(T,\cdot)}_{\D1}\le R,&\forall t\in \left[\mathscr{T},\mathscr{T}+T\right]
		\end{array} 
		\right\}.
	\end{align*}
	We aim to to show that Problem \eqref{Main Equation}-\eqref{Initial condition} has a unique mild solution that belongs to $ \mathbb{E} $, then it is the extension of $ u $ to $ I $. To this end, we take an element $ \overline{w} $ of $ \mathbb{E} $ and define the sequence $ \{w_n\}_{n=1}^\infty $ as
	\begin{align*}
		w_{1}(x,t)&:=\mathscr{S}(t)u_0(x)+\int_{0}^{t}\mathscr{R}(t-\tau)H(\overline{w}(x,\tau))\d\tau ,\\
		w_{n+1}(x,t)&:=\mathscr{S}(t)u_0(x)+\int_{0}^{t}\mathscr{R}(t-\tau)H(w_{n}(x,\tau))\d\tau.
	\end{align*}
	We can check easily for every $ t\in[0,\mathscr{T}] $, that
	\begin{align*}
		w_1(x,t)=w_2(x,t)=...=w_n(x,t)=u(x,t),\quad\forall x\in\Omega.
	\end{align*}
	On the contrary, when $ t\in[\mathscr{T},\mathscr{T}+T] $, a simple calculation shows that
	\begin{align*}
		&\norm{u_1(t)-u(\mathscr{T})}_{\D1}\\
		\le&\bnorm{\mathscr{S}(t)u_0-\mathscr{S}(\mathscr{T})u_0}_{\D1}\nonumber\\
		&+\int_{\mathscr{T}}^{t}\bnorm{\mathscr{R}(t-\tau)H(w_{n}(x,\tau))}_{\D1}\d\tau\\
		&+\int_{0}^{\mathscr{T}}\bnorm{\mathscr{R}(t-\tau)H(w_{n}(x,\tau))-\mathscr{R}(\mathscr{T}-\tau)H(w_{n}(x,\tau))}_{\D1}\d\tau.\nonumber
	\end{align*}
	Similar to Lemma \ref{Sect3_continuous}, we derive
	\begin{align*}
		\bnorm{\mathscr{S}(t)u_0-\mathscr{S}(\mathscr{T})u_0}_{\D1}\le\alpha^{-1}\mathcal{M}(t-\mathscr{T})^\alpha\norm{u_0}_{\D1}\le\alpha^{-1}\mathcal{M}T^\alpha\norm{u_0}_{\D1}
	\end{align*}
	and 
	\begin{align*}
		&\int_{0}^{\mathscr{T}}\bnorm{\mathscr{R}(t-\tau)H(w_{n}(x,\tau))-\mathscr{R}(\mathscr{T}-\tau)H(w_{n}(x,\tau))}_{\D1}\d\tau\\
		\le&\frac{\mathcal{M}\overline{\mathscr{C}}_1C_2\big|t^\alpha-(t-\mathscr{T})^\alpha-\mathscr{T}^\alpha\big|}{\alpha(1-\alpha)\theta^{\frac{1}{2}}_{1}}\Big(\norm{w_n}^2_{C(I\to\D1)}+\norm{w_n}_{C(I\to\D1)}\Big).
	\end{align*}
	In addition, the following estimate also holds
	\begin{align*}
		&\int_{\mathscr{T}}^{t}\norm{\mathscr{R}(t-\tau)H(w_{n}(x,\tau))}_{\D1}\d\tau\\
		\le&\overline{\mathscr{C}}_1C_2\int_{\mathscr{T}}^{t}(t-\tau)^{\alpha-1}\Big(\norm{w_n}^2_{\D1}+\norm{w_n}_{\D1}\Big)\d\tau\nonumber\\
		\le&\frac{\overline{\mathscr{C}}_1C_2(t-\mathscr{T})^\alpha}{\alpha}\Big(\norm{w_n}^2_{C(I\to\D1)}+\norm{w_n}_{C(I\to\D1)}\Big)\\
		\le&\frac{\overline{\mathscr{C}}_1C_2T^\alpha}{\alpha}\Big(\norm{w_n}^2_{C(I\to\D1)}+\norm{w_n}_{C(I\to\D1)}\Big).
	\end{align*}
	From the above estimates, by taking $ T $ small enough, one has
	\begin{align*}
		\norm{u_1(t)-u(\mathscr{T})}_{\D1}\le R.
	\end{align*}
	It follows that $ w_1 $ is in $ \mathbb{E} $. The same arguments show that if $ w_n\in\mathbb{E} $ for $ n\ge1 $, then $ w_{n+1} $ also belongs to $ \mathbb{E} $. To complete the proof, we have to show that $ \{w_n\}_{n=1}^\infty $ is a Cauchy sequence in $ \mathbb{E} $. Indeed, for $ w_{n-1},w_n\in\mathbb{E} $ and $ t\ge\mathscr{T} $, we obtain 
	\begin{align*}
		&\norm{w_{n+1}(t)-w_n(t)}_{\D1}\\
		\le&\int_{\mathscr{T}}^{t}\bnorm{\mathscr{R}(t-\tau)H(w_{n}(\tau))-w_{n-1}(x,\tau)}_{\D1}\d\tau\nonumber\\
		\le&\int_{\mathscr{T}}^{t}\overline{\mathscr{C}}_1C_2(t-\tau)^{\alpha-1}\Bigg(\norm{w_{n-1}(\tau)}_{\D1}+\norm{w_{n}(\tau)}_{\D1}+1\Bigg)\norm{w_{n}(\tau)-w_{n-1}(\tau)}_{\D1}\d\tau\nonumber\\
		\le&2\overline{\mathscr{C}}_1C_2T^\alpha\alpha^{-1}\Bigg(\left(\norm{u(\mathscr{T})}_{\D1}+ R\right)+1\Bigg)\norm{w_{n}-w_{n-1}}_{C(I\to\D1)},
	\end{align*}
	provided that for any $ w\in\mathbb{E} $, by the triangle inequality, the following estimate holds for any $ t\in[\mathscr{T},\mathscr{T}+T] $
	\begin{align*}
		\norm{w(t)}_{\D1}\le\norm{u(\mathscr{T})}_{\D1}+ R.
	\end{align*}  
	Hence, if $ T $ is sufficiently small, we obtain
	\begin{align*}
		\norm{w_{n+1}-w_{n}}_{C(I\to\D1)}\le\frac{3}{4}\norm{w_{n}-w_{n-1}}_{C(I\to\D1)},
	\end{align*}
	the same arguments as \eqref{Cauchy-Picard argument} show that $ \{w_n\}_{n=1}^\infty $ is a Cauchy sequence in $ \mathbb{E} $. As we said before, $ \mathbb{E} $ is a Banach space, then we can find a limit function of $ \{w_n\}_{n=1}^\infty $ which is the unique mild solution to Problem \eqref{Main Equation}-\eqref{Initial condition} on $ [0,\mathscr{T}+T]. $ The proof is completed.
\end{proof}
\begin{theorem}(Global and finite time blowup solution)
	Let $ u_0 $ be in $ \D1 $ and $ u $ be the unique mild solution to Problem \eqref{Main Equation}-\eqref{Initial condition}. If we define 
	\begin{align*}
		\mathscr{T}_{\max}:=\sup\Big\{\mathscr{T}>0~\big|~u\text{ exists uniquely on }[0,\mathscr{T}]\Big\},
	\end{align*}
	then either $ \mathscr{T}_{\max}=\infty $ ~or
	$ \lim\limits_{t\to\mathscr{T}_{\max}^-}\norm{u(t)}_{\D1}=\infty $.
\end{theorem}
\begin{proof}
	Assume that the maximal time point $ \mathscr{T}_{\max} $ is finite and $ I=[0,\mathscr{T}) $. Let us make a contradict assumption that there exists a finite constant $ \mathscr{M} $ such that
	\begin{align}
		\norm{u(t)}_{\D1}\le\mathscr{M},\quad\forall t\in[0,\mathscr{T}_{\max}).
	\end{align}  
	Let $ \{t_n\}_{n=1}^\infty $ be a sequence in $ I $ satisfying $ t_n\xrightarrow{n\rightarrow\infty}\mathscr{T}_{\max} $. We aim to show that $ \{u_n\}_{n=1}^\infty=\{u(t_n)\}_{n=1}^\infty $ is a Cauchy sequence in $ \D1 $. In fact, for $ t_n<t_m $, we find that
	\begin{align*}
		&\norm{u_{m}-u_{n}}_{\D1}\\
		\le&\bnorm{\mathscr{S}(t_m)u_0-\mathscr{S}(t_n)u_0}_{\D1}\nonumber\\
		&+\int_{t_n}^{t_m}\bnorm{\mathscr{R}(t_m-\tau)H(u(\tau))}_{\D1}\d\tau\nonumber\\
		&+\int_{0}^{t_n}\bnorm{\mathscr{R}(t_m-\tau)H(u(\tau))-\mathscr{R}(t_n-\tau)H(u(\tau))}_{\D1}\d\tau=:\sum_{j\in\{1,2,3\}}\mathcal{I}_j.
	\end{align*}
	On one hand, since $ u_0 $ is in $ \D1 $, we have
	\begin{align*}
		\mathcal{I}_1\le\alpha^{-1}\mathcal{M}(t_m-t_n)^\alpha\norm{u_0}_{\D1}\le\alpha^{-1}\mathcal{M}\Big((t_m-\mathscr{T}_{\max})^\alpha+(\mathscr{T}_{\max}-t_n)^\alpha\Big)\norm{u_0}_{\D1}.
	\end{align*}
	Then, for any $ \varepsilon>0 $, we can find a number $ n_1 $ such that 
	\begin{align*}
		\mathcal{I}_1\le\frac{\varepsilon}{3},\quad\forall m,n\ge n_1.
	\end{align*}
	On the other hand, the use of Lemma \ref{Linear estimate}\textit{(ii)} shows that
	\begin{align*}
		\mathcal{I}_2&\le C_2\int_{t_n}^{t_m}(t_m-\tau)^{\alpha-1}\bnorm{H(u(\tau))}_{\Doneminus}\d\tau\nonumber\\
		&\le \overline{\mathscr{C}}_1C_2\int_{t_n}^{t_m}(t_m-\tau)^{\alpha-1}\Big(\norm{u(\tau)}^2_{\D1}+\norm{u(\tau)}_{\D1}\Big)\d\tau\\
		&\le\alpha^{-1}\big(\mathscr{M}^2+\mathscr{M}\big)\overline{\mathscr{C}}_1C_2\Big((t_m-\mathscr{T}_{\max})^\alpha+(\mathscr{T}_{\max}-t_n)^\alpha\Big).
	\end{align*}
	Then, we can also find a number $ n_2 $ such that
	\begin{align*}
		\mathcal{I}_2\le\frac{\varepsilon}{3},\quad\forall m,n\ge n_2.
	\end{align*}
	Also, it's obviously that 
\begin{align*}
&\Bnorm{\mathscr{R}(t_m-\tau)H(u(\tau))-\mathscr{R}(t_n-\tau)H(u(\tau))}^2_{\D1}\\
=&\sum_{k=1}^{\infty}\frac{\theta_{k}\mathscr{E}^2(k,t_n,\tau,t_m-t_n)}{( 1+\theta_{k})^{2}}\Binner{H(u(\tau))}{\phi_{k}}^{2}\\
\le& \sum_{k=1}^{\infty}\frac{\theta_{k}\mathscr{E}^2(k,t_n,\tau,t_m-t_n)}{( 1+\theta_{k})^{2}}\Binner{H(u(\tau))}{\phi_{k}}^{2}.
\end{align*}
Then, from \eqref{Inequality for E(k,t,tau,epsilon)}, the following inequality holds 
	\begin{align*}
		&\int_{0}^{t_n}\Bnorm{\mathscr{R}(t_m-\tau)H(u(\tau))-\mathscr{R}(t_n-\tau)H(u(\tau))}_{\D1}\d\tau\nonumber\\
		\le&\frac{\mathcal{M}\overline{\mathscr{C}}_1C_2}{(1-\alpha)}\int_{0}^{t_n}\Big((t_n-\tau)^{\alpha-1}-(t_m-\tau)^{\alpha-1}\Big)\bnorm{H(u(\tau))}_{D^{-1}(\Omega)}\d\tau\\
		\le&\frac{\mathcal{M}\big(\mathscr{M}^2+\mathscr{M}\big)\overline{\mathscr{C}}_1C_2}{\alpha(1-\alpha)}\big|t_m^\alpha-(t_m-t_n)^\alpha-t_n^\alpha\big|\nonumber\\
		\le&\frac{\mathcal{M}\big(\mathscr{M}^2+\mathscr{M}\big)\overline{\mathscr{C}}_1C_2}{\alpha(1-\alpha)}\Big(\big|t_n^\alpha-\mathscr{T}_{\max}^\alpha\big|+\big|t_n-\mathscr{T}_{\max}\big|^\alpha+\big|t_m-\mathscr{T}_{\max}\big|^\alpha +\big|t_m^\alpha-\mathscr{T}_{\max}^\alpha\big|\Big).
	\end{align*}
	The convergence of $ \{t_n\}_{n=1}^\infty $ gives us a large number $ n_3 $ such that 
	\begin{align*}
		\mathcal{I}_3\le\frac{\varepsilon}{3},\quad\forall m,n\ge n_3.
	\end{align*}
	Hence, for every $ \varepsilon>0 $, we can always find a number $ n_0=\max\left\{n_1,n_2,n_3\right\} $ that whenever $ m,n\ge n_0 $, we have
	\begin{align*}
		\norm{u_m-u_n}_{\D1}\le\varepsilon.
	\end{align*}
	Accordingly, we can conclude that $ \{u_n\}_{n=1}^\infty $ is a Cauchy sequence in $ \D1 $. Thanks to the completeness of $ \D1 $, there exists a unique function $ \overline{u} $ which is the limit of $ u_n $ as $ n $ to infinity, i.e.
	\begin{align*}
		\overline{u}=\lim\limits_{n\to\infty}u(x,t_n)=u(x,\mathscr{T}_{\max}).
	\end{align*} 
	Thus we have extended $ u $ over the interval $ [0,\mathscr{T}_{\max}] $. Now, Lemma \ref{ExtensionLemmaSect4} is available for extending $ u $ to some large time interval. This fact leads us to a contradiction. Then, we can finish our proof by making the conclusion that
	\begin{align*}
		\lim\limits_{t\to\mathscr{T}_{\max}^-}\norm{u(t)}_{\D1}=\infty,
	\end{align*} 
provided $ \mathscr{T}_{\max}<\infty. $
\end{proof}
\section{The other Lipschitz cases of the source functions}\label{Local Lipschitzian}
\noindent Throughout this section, we concern in the following form of the initial value problem \eqref{Main Equation}-\eqref{Initial condition}
\begin{align}
\begin{cases}
\dis\partial_{t}^{\alpha}u - \partial_{t}^{\alpha}\Delta u -\Delta u =H\left(u\right),\quad&\text{in}\quad\Omega\times(0, \infty),\\[0.15cm]
\dis u=0,&\text{on}\quad\partial\Omega\times(0, \infty),\\[0.15cm]
\dis u=u_{0},&\text{in}\quad\Omega\times\{0\}.
\end{cases}
\end{align}
In particular, our main results for a mild solution $ u:[0,\infty)\to\Omega $ are going to revolve around the following two hypotheses for the source functions $ H $.

\noindent \textit{$ (H_1) $ The nonlinearity of polynomial type}
\begin{flalign*}
\quad~H(u)=|u|^{p-1}u,\quad p> 2.&&
\end{flalign*}
\noindent \textit{$ (H_2) $ The nonlinearity of exponential type}
\begin{flalign*}
\quad~ H(u)=u^3e^{u^2}. &&
\end{flalign*}
\subsection{The polynomial nonlinearity}
For the hypothesis $ (H_1) $ for the source function, we aim to establish the global-in-time well-posedness result for the mild solution to Problem \eqref{Main Equation}-\eqref{Initial condition}. To this end we first provide some nonlinear control on $ H $.

\begin{lemma}\label{Nonlinear estimate Polynomial}
Let $d\geq 2$ and $p>2$. Then, if
\begin{align*}
0<\nu<\min\left\{1,\frac{d+2-2(p-1)d}{2-4(p-1)}\right\},
\end{align*}
we can find a independent constant $\mathscr{C}_2>0$ such that
\begin{align*}
	\bnorm{H(u)-H(v)}_{D^{\nu-1}(\Omega)}\leq \mathscr{C}_2\left(\norm{u}_{D^{\nu}(\Omega)}^{p-1}+\norm{v}_{D^{\nu}(\Omega)}^{p-1}\right)\norm{u-v}_{D^{\nu}(\Omega)},
\end{align*}
for every $u, v\in D^{\nu}(\Omega)$.
\end{lemma}
\begin{proof}
Firstly, for any $ q\ge1 $, H\"older's inequality gives us the following estimate
\begin{align*}
\bnorm{|u|^{p-1}(u-v)}_{L^{q}(\Omega)}\leq\norm{u-v}_{L^{2q}(\Omega)}\norm{u}_{L^{2(p-1)q}(\Omega)}^{p-1}.
\end{align*}	
Then, if we choose $ q=\frac{2d}{d-2(\nu-1)}>1 $, Lemma \ref{Fractional Sobolev embeddings} shows that the Lebesgue space $ L^{q}(\Omega) $ embeds into the Hilbert space $ D^{\nu-1}(\Omega) $. Also, by another use of Lemma \ref{Fractional Sobolev embeddings} and the assumptions for $d, p,\nu$ and $ \Omega $, we have	
\begin{align*}
D^{\nu}(\Omega)\hookrightarrow L^{2(p-1)q}(\Omega)\hookrightarrow L^{2q}(\Omega)\hookrightarrow L^{q}(\Omega).
\end{align*}
Combining the above result with the fact $(a+b)^{m}\leq 2^m (a^{m}+b^{m}),~ a, b, m>0$ and the triangle inequality, for any $ u,v\in\Dnu $ we have
\begin{align*}
\bnorm{H(u)-H(v)}_{D^{\nu-1}(\Omega)}\le 2^q C_{\Omega}\left(\norm{u}_{D^{\nu}(\Omega)}^{p-1}+\norm{v}_{D^{\nu}(\Omega)}^{p-1}\right)\norm{u-v}_{D^{\nu}(\Omega)},
\end{align*}
where $ C_{\Omega} $ comes from the embeddings. The proof is completed.
\end{proof}
\begin{definition}
Let $ I=[0,\mathscr{T}] $ and $ \sigma,M>0 $. Then, we define a subset of $ C\left(I\to\Dnu\right) $ as follows
\begin{align*}
\mathbb{X}=\mathbb{X}(\sigma,M,\nu):=\left\{w\in C(I \rightarrow D^{\nu}(\Omega))~\Big|~w(0)=u_0~\text{ and }\sup_{t\in I\setminus \{0\}}t^{\sigma}\norm{w(t)}_{D^{\nu}(\Omega)}< M\right\}.
\end{align*}
\end{definition}
\begin{remark}
Observe that $ \mathbb{X} $ is a Banach space involving the norm
\begin{align*}
\norm{w}_{\mathbb{X}}=\sup_{t\in I\setminus \{0\}} t^{\sigma}\norm{w(t)}_{D^{\nu}(\Omega)}.
\end{align*}
\end{remark}
\begin{theorem}\label{Polynomial main theorem}
Let $\alpha\in(0,1)$ and $p>2$  such that $  \frac{p}{p-1}<\alpha^{-1}$.  For a given datum $u_{0}\in D^{\nu}(\Omega)$ sufficiently small and $I =[0, \mathscr{T}]$. Problem $\eqref{Main Equation}$-$\eqref{Initial condition}$  admits a unique mild solution that belongs to the space $ \mathbb{X}=\mathbb{X}\left(\frac{\alpha}{p-1},2C_1\norm{u_0}_{\Dnu},\nu\right)$.
\end{theorem}

\begin{proof}
In the same spirit with Section \ref{GBBM Eq}, we consider again the sequence $\{w_{n}\}_{n=1}^{\infty}$ established by
\begin{align*}
w_{1}(x,t)&:=\mathscr{S}(t)u_0(x) ,\\
w_{n+1}(x,t)&:=\mathscr{S}(t)u_0(x)+\int_{0}^{t}\mathscr{R}(t-\tau)H(w_{n}(x,\tau))\d\tau.
\end{align*}

By analogous argument as Lemma \ref{Sect3_continuous}, we can prove that $\{w_{n}\}_{n=1}^{\infty}$ is a subset of $C (I \rightarrow D^{\nu}(\Omega))$. So, we need only to focus on deriving that $\{w_{n}\}_{n=1}^{\infty}$ is a convergent sequence in $ \mathbb{X}=\mathbb{X}\left(\alpha\mu,2C_1\norm{u_0}_{\Dnu},\nu\right) $, where $\mu=\frac{1}{p-1}$. The Picard iteration process for this proof includes three main steps.\\

\noindent$\bullet$ \textbf{Step 1}. We show that $w_{1}$ is in $\mathbb{X}$.

\noindent Indeed, since $\mu\in(0,1)$ and $u_{0}\in D^{\nu}(\Omega)$, thanks to Lemma \ref{Linear estimate}({\textit{i}}), we obtain
\begin{align*}
\norm{w_{1}(t)}_{D^{\nu}(\Omega)}\leq C_1t^{-\alpha\mu}\norm{u_{0}}_{D^{\nu}(\Omega)},
\end{align*}
which implies
\begin{align*}
t^{\alpha\mu}\norm{w_{1}(t)}_{D^{\nu}(\Omega)}\le C_1\norm{u_{0}}_{D^{\nu}(\Omega)}.
\end{align*}

\noindent$\bullet$ \textbf{Step 2}. Assume that $w_{n}$ belongs to $\mathbb{X}$ for any $n\in \mathbb{N}$, then $w_{n+1}\in \mathbb{X}$.

\noindent Indeed, Lemma \ref{Linear estimate}(\textit{ii}) provides us the following estimate
\begin{align}
\norm{w_{n+1}(t)-w_{1}(t)}_{D^{\nu}(\Omega)}\leq C_2\int_{0}^{t}(t-\tau)^{\alpha-1}\Vert H(w_{n}(\tau))\Vert_{D^{\nu-1}(\Omega)}\d\tau.
\end{align}
Using Lemma \ref{Nonlinear estimate Polynomial} with $v=0$ and the fact that $w_{n}\in \mathbb{X}$, we have
\begin{align*}
\norm{w_{n+1}(t)-w_{1}(t)}_{D^{\nu}(\Omega)}&\leq \mathscr{C}_2C_2\int_{0}^{t}(t-\tau)^{\alpha-1}\norm{w_{n}(\tau)}_{D^{\nu}(\Omega)}^{p}\d\tau\nonumber\\
&\leq \mathscr{C}_2C_2\left(\int_{0}^{t}(t-\tau)^{\alpha-1}\tau^{-p\alpha\mu}\d\tau\right)\norm{w_{n}}_{\mathbb{X}}^{p}\\
&\leq \mathscr{C}_2C_2\left(\int_{0}^{t}(t-\tau)^{\alpha-1}\tau^{-p\alpha\mu}\d\tau\right)\left(2C_1\norm{u_{0}}_{D^{\nu}(\Omega)}\right)^{p}.\nonumber
\end{align*}
Also, from our assumptions, we find that
\begin{align*}
p\alpha\mu<1 \quad\text{and}\quad 1-p\mu+\mu=0, 
\end{align*}
which tells that
\begin{align*}
t^{\alpha\mu}  \int_{0}^{t}(t-\tau)^{\alpha-1}\tau^{-p\alpha\mu}\d\tau=\beta(\alpha, 1-p\alpha\mu),
\end{align*}
where the Beta function is defined in Definition \ref{Gamma-Beta defs}. From the standpoint of this result, we infer
\begin{align}
t^{\alpha\mu}\norm{w_{n+1}(t)}_{D^{\nu}(\Omega)}\leq C_1\norm{u_{0}}_{D^{\nu}(\Omega)}+\mathscr{C}_2C_2\beta(\alpha, 1-p\alpha\mu)\left(2C_1\norm{u_{0}}_{D^{\nu}(\Omega)}\right)^{p}.
\end{align}
Hence, the statement that $w_{n+1}\in \mathbb{X}$ follows from the small  data $u_{0} D^{\nu}(\Omega)$, then Step 2 is completed.

\noindent$\bullet$ \textbf{Step 3}. We show that $\{w_{n}\}_{n=1}^{\infty}$ is a Cauchy sequence.

\noindent Suppose that $w_{n}$ and $ w_{n-1}$ are in $\mathbb{X}$. Our techniques are not too different from Step 2. Indeed, we have
\begin{align}\label{Polynomial Cauchy-Picard sp 1}
&\norm{w_{n+1}(t)-w_{n}(t)}_{D^{\nu}(\Omega)}\nonumber\\
\leq& \mathscr{C}_2C_2\int_{0}^{t}(t-\tau)^{\alpha-1}\left(\norm{w_{n}(\tau)}_{D^{\nu}(\Omega)}^{p-1}+\norm{w_{n-1}(\tau)}_{D^{\nu}(\Omega)}^{p-1}\right)\norm{w_{n}(\tau)-v_{n-1}(\tau)}_{D^{\nu}(\Omega)}\d\tau.
\end{align}
Since $\{w_{n}\}_{n=1}^{\infty}\in\mathbb{X}$, multiplying the both sides of \eqref{Polynomial Cauchy-Picard sp 1} by $t^{\alpha\mu}$, we have
\begin{align}\label{Polynomial Cauchy-Picard sp 2}
&t^{\alpha\mu}  \norm{w_{n+1}(t)-w_{n}(t)}_{D^{\nu}(\Omega)}\nonumber\\
\leq& 2\mathscr{C}_2C_2\left(t^{\alpha\mu}\int_{0}^{t}(t-\tau)^{\alpha-1}\tau^{-p\alpha\mu}\d\tau\right)\left(2C_1\norm{u_{0}}_{D^{\nu}(\Omega)}\right)^{p-1}\norm{w_{n}(t)-w_{n-1}(t)}_{\mathbb{X}}\nonumber\\
\le&2\mathscr{C}_2C_2\beta(\alpha, 1-p\alpha\mu)\left(2C_1\norm{u_{0}}_{D^{\nu}(\Omega)}\right)^{p-1}\norm{w_{n}(t)-w_{n-1}(t)}_{\mathbb{X}}.
\end{align}
Then, we can take the supremum over the interval $ I\setminus\{0\} $ on both sides of \eqref{Polynomial Cauchy-Picard sp 2} to obtain
\begin{align*}
\norm{w_{n+1}(t)-w_{n}(t)}_{\mathbb{X}}\le2\mathscr{C}_2C_2\beta(\alpha, 1-p\alpha\mu)\left(2C_1\norm{u_{0}}_{D^{\nu}(\Omega)}\right)^{p-1}\norm{w_{n}(t)-w_{n-1}(t)}_{\mathbb{X}}.
\end{align*} 
Therefore, if $ \norm{u_0}_{\Dnu} $ is sufficiently small, $ \{w_{n}\}_{n=1}^{\infty}$ is a Cauchy sequence in $\mathbb{X} $.

From the above three steps, we can use again the arguments performed as in the proofs of Section \ref{GBBM Eq} to find the limit function $ u $ of $ \{w_n\}_{n=1}^\infty $ which is the unique mild solution of Problem \eqref{Main Equation}-\eqref{Initial condition}. The proof is completed.
\end{proof}
\subsection{Exponential nonlinearity}
This subsection concerns the global well-posedness of Problem \eqref{Main Equation}-\eqref{Initial condition} with the source function $ H(u)=u^3e^{u^2} $. Here for any $ w,v\in\R $, we have
\begin{align*}
|H(w)-H(v)|\le\mathscr{C}_3\left|w-v\right|\left(w^2e^{w^2}+v^2e^{v^2}\right),
\end{align*} 
for some indepent $ \mathscr{C}_3>0. $ We recall the following theorem  in \cite[Theorem 2]{Trudinger} or \cite[Section 1]{Martinazzi} which is used to control the solution operator in the framework of Orlicz space.

\begin{theorem}\label{Embedding into Orlicz} Let $ \Omega\subset\R^2 $ be a cone domain. Then, $ W_0^{1,2}(\Omega) $ embeds continuously into the Orlicz space $ \Lxi $.
\end{theorem}

\begin{lemma}\label{Exp_Contraction_estimates}
Let $ \alpha\in\left(0,\frac{2}{3}\right), $ $ \Omega\subset\R^2 $ be a bounded domain with sufficiently smooth boundary, and $ M $ be a finite constant. Suppose that $ I=[0,\mathscr{T}] $ and $ w,v $ are functions in $ C\left(I\to\Lxi\right) $ such that
\begin{align*}
\max\left(\sup_{t\in I}\norm{w(t)}_{\Lxi},\sup_{t\in I}\norm{t^{\frac{\alpha}{2}}w(t)}_{\Lxi},\sup_{t\in I}\norm{w(t)}_{\Lxi},\sup_{t\in I}\norm{t^{\frac{\alpha}{2}}v(t)}_{\Lxi}\right)<M.
\end{align*}
If $ M $ is small enough, we can derive the following inequalities\\
\begin{align}\label{Inequality 1}
&\norm{\int_{0}^{t}\mathscr{R}(t-\tau)\big(H(w(\tau))-H(v(\tau))\big)\d\tau}_{\Lxi}\nonumber\\
\le&\frac{2\mathscr{C}_3\mathscr{T}^\alpha}{\alpha}\norm{w-v}_{\Lxi}M^2\left(\left(\Gamma\left(3\right)\right)^{\frac{1}{2}}+\left(\Gamma\left(4\right)\right)^{\frac{1}{2}}\left(6M^2\right)^{\frac{1}{6}}\right)
\end{align}
and
\begin{align}\label{Inequality 2}
&t^{\frac{\alpha}{2}}\norm{\int_{0}^{t}\mathscr{R}(t-\tau)\big(H(w(\tau))-H(v(\tau))\big)\d\tau}_{\Lxi}\nonumber\\
\le&\frac{2\mathscr{C}_3\norm{w-v}_{\Lxi}}{\big(\beta\left(\alpha,1-\frac{3\alpha}{2} \big)\right)^{-1}}M^2\left(\left(\Gamma\left(3\right)\right)^{\frac{1}{2}}+\left(\Gamma\left(4\right)\right)^{\frac{1}{2}}\left(6M^2\right)^{\frac{1}{6}}\right).
\end{align}
\end{lemma}
\begin{proof}
Assume that $ w$ and $v $ are functions in $ \Lxi $. Then, we can use Theorem \ref{Embedding into Orlicz} to obtain 
\begin{align*}
\norm{\int_{0}^{t}\mathscr{R}(t-\tau)\left(H(w(\tau))-H(v(\tau))\right)\d\tau}_{\Lxi}\le\int_{0}^{t}\Bnorm{\mathscr{R}(t-\tau)\left(H(w(\tau))-H(v(\tau))\right)}_{\D1}\d\tau.
\end{align*}
Using Lemma \ref{Linear estimate}, we have
\begin{align}\label{Contraction estimate1 sect5}
\norm{\int_{0}^{t}\mathscr{R}(t-\tau)\left(H(w(\tau))-H(v(\tau))\right)\d\tau}_{\Lxi}\le\int_{0}^{t}(t-\tau)^{\alpha-1}\Bnorm{H(w(\tau))-H(v(\tau))}_{\L2}\d\tau.
\end{align}
To continuous, we make a nonlinear estimate for the source function with $ w=w(t),v=v(t),~t\in I $, as follows
\begin{align}\label{Exp nonlinear estimate sp1}
&\norm{H(w)-H(v)}_{\L2}\\
\le&\mathscr{C}_3\norm{(w-v)(u^2+v^2)}_{\L2}+\mathscr{C}_3\sum_{u\in\{w,v\}}\norm{\left|w-v\right||u|^2\left(e^{u^2}-1\right)}_{\L2}\nonumber\\
\le&\mathscr{C}_3\sum_{u\in\{w,v\}}\left(\norm{w-v}_{L^4(\Omega)}\norm{u}^2_{L^4(\Omega)}+\norm{w-v}_{L^6(\Omega)}\norm{u}^2_{L^6(\Omega)}\norm{e^{u^2}-1}_{L^6(\Omega)}\right),\nonumber
\end{align}
here, we have used H\"older's inequality. If $ \norm{u}_{\Lxi}<(\frac{1}{6})^{\frac{1}{2}} $,    it follows from the techniques used in \cite[Lemma 3.2]{Ioku} that
\begin{align}\label{Exp nonlinear estimate sp2}
\norm{e^{u^2}-1}^6_{L^6(\Omega)}&\le\int_{\Omega}\left(e^{6u^2(x)}-1\right){\d} x\le\int_{\Omega}\left(\exp\left({\frac{6\norm{u}^2_{\Lxi}u^2(x)}{\norm{u}^2_{\Lxi}}}\right)-1\right){\d} x\le6\norm{u}^2_{\Lxi},
\end{align}
where we have used the fact that 
\begin{align*}
\left\{\kappa\in\R~\Big|~\kappa>0,\int_{\Omega}\Xi\left(\frac{|w(x)|}{\kappa}\right)\d x\le1\right\}=\left[\norm{w}_{\Lxi},\infty\right).
\end{align*}
According to \eqref{Exp nonlinear estimate sp1}, \eqref{Exp nonlinear estimate sp2} and Lemma \ref{Orlicz embeds into Lp}, we find that
\begin{align}\label{Nonlinear estimate exponential}
\norm{H(w)-H(v)}_{\L2}&\le\mathscr{C}_3\sum_{u\in\{w,v\}}\norm{w-v}_{\Lxi}\norm{u}^2_{\Lxi}\left(\left(\Gamma\left(3\right)\right)^{\frac{1}{2}}+\left(\Gamma\left(4\right)\right)^{\frac{1}{2}}\left(6\norm{u}^2_{\Lxi}\right)^{\frac{1}{6}}\right).
\end{align}
Inclusion inequalities \eqref{Contraction estimate1 sect5} and \eqref{Nonlinear estimate exponential} allow us to deduce the inequality \eqref{Inequality 1}. Furthermore, by using \eqref{Inequality 1} and
\begin{align*}
	t^{\frac{\alpha}{2}}  \int_{0}^{t}(t-\tau)^{\alpha-1}\tau^{-\frac{3\alpha}{2}}\d\tau=\beta\left(\alpha, 1-\frac{3\alpha}{2}\right),
\end{align*}
we get the result \eqref{Inequality 2}.
\end{proof}
\begin{theorem}
Let $ \Omega $ be a bounded domain of $ \R^2 $ with sufficiently smooth boundary $ \partial\Omega $ and $ H(u)=u^3e^{u^2} $. Assume that $ I=[0,\mathscr{T}] $ and $ u_0\in\D1 $ is small enough. Then, there exists a unique mild solution of Problem \eqref{Main Equation}-\eqref{Initial condition} in $ C\left(I\to\Lxi\right). $
\end{theorem}
\begin{proof}
We start by considering again the sequence $\{w_{n}\}_{n=1}^{\infty}$ as in Theorem \ref{Polynomial main theorem} and the following space
\begin{align*}
    \mathbb{U}:=\left\{w\in C\left(I\to\Lxi\right)~\Big|~\norm{w}_{    \mathbb{U}}<2C_1\norm{u_0}_{\D1}\right\},
\end{align*}
where, the $ \mathbb{U} $-norm is given by
\begin{align*}
\norm{w}_{    \mathbb{U}}:=\max\left\{\sup_{t\in I}\norm{w}_{\Lxi},\sup_{t\in I}t^{\frac{\alpha}{2}}\norm{w}_{\Lxi}\right\}.
\end{align*}
Using Theorem \ref{Embedding into Orlicz} and considering Lemma \ref{Linear estimate} in two cases: $ \mu=0 $ and $ \mu=\frac{1}{2} $, we have
\begin{align*}
\begin{cases}
\dis\norm{w_1(t)}_{\Lxi}\le\norm{\mathscr{S}(t)u_0}_{\D1}\le C_1\norm{u_0}_{\D1},\qquad&\mu=0,~t\in I\\[0.2cm]
\dis t^{\frac{\alpha}{2}}\norm{w_1(t)}_{\Lxi}\le t^{\frac{\alpha}{2}}\norm{\mathscr{S}(t)u_0}_{\D1}\le C_1\norm{u_0}_{\D1},&\mu=\frac{1}{2},~t\in I.
\end{cases}
\end{align*}
This result yields that $ w_1\in  \mathbb{U}. $ Next we assume that $ w_n $ is in $   \mathbb{U} $, then we can show that $ w_{n+1} $ is also in $ \mathbb{U} $ In fact, for given small data $ u_0 $, set
\begin{align*}
\mathscr{T}_*:=\left(\frac{32\mathscr{C}_3\norm{u_0}_{\D1}^2\left(\left(\Gamma\left(3\right)\right)^{\frac{1}{2}}+\left(\Gamma\left(4\right)\right)^{\frac{1}{2}}\left(6(2\norm{u_0}_{\D1})^2\right)^{\frac{1}{6}}\right)}{\alpha C_1}\right)^{\frac{-1}{\alpha}}.
\end{align*}
On the one hand, when $t\le\mathscr{T}_*$, we apply Lemma \ref{Exp_Contraction_estimates} with $ v\equiv0 $ to find that
\begin{align*}
\norm{\int_{0}^{t}\mathscr{R}(t-\tau)H(w_{n}(\tau))\d\tau}_{\Lxi}&\le\frac{2\mathscr{C}_3\mathscr{T}_*^\alpha}{\alpha(2\norm{u_0}_{\D1})^{-3}}\left(\left(\Gamma\left(3\right)\right)^{\frac{1}{2}}+\left(\Gamma\left(4\right)\right)^{\frac{1}{2}}\left(6(2\norm{u_0}_{\D1})^2\right)^{\frac{1}{6}}\right)\nonumber\\
&=2^{-1}C_1\norm{u_0}_{\D1}.
\end{align*}
On the other hand, if $ t$ is larger than $ \mathscr{T}_* $ defined above, we can deduce that
\begin{align*}
&\norm{\int_{\mathscr{T}_*}^{t}\mathscr{R}(t-\tau)H(w_{n}(\tau))\d\tau}_{\Lxi}\\
\le&\mathscr{T}_*^{-\frac{\alpha}{2}}t^{\frac{\alpha}{2}}\norm{\int_{0}^{t}\mathscr{R}(t-\tau)H(w_{n}(\tau))\d\tau}_{\Lxi}\nonumber\\
\le&\frac{2\mathscr{C}_3\mathscr{T}_*^{-\frac{\alpha}{2}}(2\norm{u_0}_{\D1})^3\left(\left(\Gamma\left(3\right)\right)^{\frac{1}{2}}+\left(\Gamma\left(4\right)\right)^{\frac{1}{2}}\left(6(2\norm{u_0}_{\D1})^2\right)^{\frac{1}{6}}\right)}{\left(\beta\left(\alpha,1-\frac{3\alpha}{2} \right)\right)^{-1}}\nonumber\\
=&\frac{128\norm{u_0}_{\D1}^4\left(\mathscr{C}_3\left(\left(\Gamma\left(3\right)\right)^{\frac{1}{2}}+\left(\Gamma\left(4\right)\right)^{\frac{1}{2}}\left(6(2\norm{u_0}_{\D1})^2\right)^{\frac{1}{6}}\right)\right)^{\frac{3}{2}}}{\left(\beta\left(\alpha,1-\frac{3\alpha}{2} \right)\right)^{-1}}.
\end{align*}
Hence, we can choose $ \norm{u_0}_{\D1} $ small enough to obtain
\begin{align}\label{Sect5_T* to t}
\norm{\int_{\mathscr{T}_*}^{t}\mathscr{R}(t-\tau)H(w_{n}(\tau))\d\tau}_{\Lxi}\le\frac{C_1\norm{u_0}_{\D1}}{2}.
\end{align}
Combining all of the above arguments, whether $ t\le\mathscr{T}_* $ or $ t>\mathscr{T}_*, $ we can see that
\begin{align}\label{Sect5_Invariant_sp_1}
\norm{\int_{0}^{t}\mathscr{R}(t-\tau)H(w_{n}(\tau))\d\tau}_{\Lxi}\le C_1\norm{u_0}_{\D1}.
\end{align}
In addition, apply again Lemma \ref{Contraction estimate1 sect5}, for any $ w_n\in\mathbb{U} $, we obtain
\begin{align}\label{Sect5_Invariant_sp_2}
&t^{\frac{\alpha}{2}}\norm{\int_{\mathscr{T}_*}^{t}\mathscr{R}(t-\tau)H(w_{n}(\tau))\d\tau}_{\Lxi}\nonumber\\
\le&\frac{2\mathscr{C}_3(2C_1\norm{u_0}_{\D1})^3}{\left(\beta\left(\alpha,1-\frac{3\alpha}{2} \right)\right)^{-1}}\left(\left(\Gamma\left(3\right)\right)^{\frac{1}{2}}+\left(\Gamma\left(4\right)\right)^{\frac{1}{2}}\left(6(2C_1\norm{u_0}_{\D1})^2\right)^{\frac{1}{6}}\right).
\end{align} 
Therefore, for any $ w_n\in\mathbb{U} $, we deduce the following two claims.

\noindent\textbf{Claim 1}. Combining the conclusions that $ w_1\in\mathbb{U} $ and the inequality \eqref{Sect5_Invariant_sp_1} gives us 
\begin{align*}
\norm{w_{n+1}(t)}_{\Lxi}\le\norm{\mathscr{S}(t)u_0}_{\Lxi}+\norm{\int_{0}^{t}\mathscr{R}(t-\tau)H(w_{n}(\tau))\d\tau}_{\Lxi}\le2C_1\norm{u_0}_{\D1}.
\end{align*}

\noindent\textbf{Claim 2}. Similarly, since $ w_1\in\mathbb{U} $ and the estimate \eqref{Sect5_Invariant_sp_2} holds, we have
\begin{align*}
t^{\frac{\alpha}{2}}\norm{w_{n+1}(t)}_{\Lxi}\le t^{\frac{\alpha}{2}}\norm{\mathscr{S}(t)u_0}_{\Lxi}+t^{\frac{\alpha}{2}}\norm{\int_{0}^{t}\mathscr{R}(t-\tau)H(w_{n}(\tau))\d\tau}_{\Lxi}\le2C_1\norm{u_0}_{\D1},
\end{align*}
as long as $ \norm{u_0}_{\D1} $ is sufficiently small.

These claims show that for any $ t\in I $, $ \bnorm{w_{n+1}}_{\mathbb{U}(t)} $ is less than or equals to $ 2C_1\norm{u_0}_{\D1} $. It means that $ w_{n+1} $ belongs to $ \mathbb{U} $ whenever $ w_n $ is in $ \mathbb{U} $, provided that the continuity of $ w_{n+1} $ is inferred by Lemma \ref{Sect3_continuous}. The remaining work is to show that $ \{w_n\}_{n=1}^{\infty} $ is a Cauchy sequence with respect to the $ \mathbb{U} $ norm. Based on this result, we can easily obtain the unique mild solution $ u $ of Problem \eqref{Main Equation}-\eqref{Initial condition} which is the limit function of the sequence $ \{w_n\}_{n=1}^{\infty} $. To this end, we take two elements of the sequence, $ w_{n-1},w_n\in\mathbb{U} $. Analogous to the way we find the estimate \eqref{Sect5_Invariant_sp_1}, we set
\begin{align*}
\mathscr{T}_{**}:=\left(\frac{64\mathscr{C}_3\norm{u_0}^2_{\D1}\left(\left(\Gamma\left(3\right)\right)^{\frac{1}{2}}+\left(\Gamma\left(4\right)\right)^{\frac{1}{2}}\left(6(2\norm{u_0}_{\D1})^2\right)^{\frac{1}{6}}\right)}{\alpha C_1}\right)^{\frac{-1}{\alpha}}.
\end{align*}
On the one hand, for any $ t>\mathscr{T}_{**} $, we have
\begin{align*}
&\Bgnorm{\int_{0}^{t}\mathscr{R}(t-\tau)\big(H(w_{n}(\tau))-H(w_{n-1}(\tau))\big)\d\tau}_{\Lxi}&\nonumber\\
\le&\underbrace{\int_{0}^{\mathscr{T}_{**}}\Bnorm{\mathscr{R}(t-\tau)\big(H(w_{n}(\tau))-H(w_{n-1}(\tau))\big)}_{\Lxi}\d\tau}_{\mathrm{I}_1}&\\
&+\underbrace{\int_{\mathscr{T}_{**}}^{t}\Bnorm{\mathscr{R}(t-\tau)\big(H(w_{n}(\tau))-H(w_{n-1}(\tau))\big)}_{\Lxi}\d\tau}_{\mathrm{I}_2}.&
\end{align*}
In view of Lemma \ref{Exp_Contraction_estimates}, one finds that
\begin{flalign*}
\mathrm{I}_1&\le\frac{2\mathscr{C}_3\mathscr{T}_{**}^\alpha}{\alpha(2\norm{u_0}_{\D1})^{-2}}\left(\left(\Gamma\left(3\right)\right)^{\frac{1}{2}}+\left(\Gamma\left(4\right)\right)^{\frac{1}{2}}\left(6(2\norm{u_0}_{\D1})^2\right)^{\frac{1}{6}}\right)\norm{w_n(t)-w_{n-1}(t)}_{\mathbb{U}}&\nonumber\\
&\le\frac{1}{4}\norm{w_n(t)-w_{n-1}(t)}_{\mathbb{U}}.&
\end{flalign*}
Applying Lemma \ref{Exp_Contraction_estimates} in the same way as in \eqref{Sect5_T* to t}, we have
\begin{flalign*}
\mathrm{I}_2&\le\frac{2\mathscr{C}_3\mathscr{T}_{**}^{-\frac{\alpha}{2}}(2\norm{u_0}_{\D1})^2\left(\left(\Gamma\left(3\right)\right)^{\frac{1}{2}}+\left(\Gamma\left(4\right)\right)^{\frac{1}{2}}\left(6(2\norm{u_0}_{\D1})^2\right)^{\frac{1}{6}}\right)}{\left(\beta\left(\alpha,1-\frac{3\alpha}{2} \right)\right)^{-1}}\norm{w_n(t)-w_{n-1}(t)}_{\mathbb{U}}&\nonumber\\
&=\frac{64\norm{u_0}_{\D1}^3\left(\mathscr{C}_3\left(\left(\Gamma\left(3\right)\right)^{\frac{1}{2}}+\left(\Gamma\left(4\right)\right)^{\frac{1}{2}}\left(6(2\norm{u_0}_{\D1})^2\right)^{\frac{1}{6}}\right)\right)^{\frac{3}{2}}}{\left(\beta\left(\alpha,1-\frac{3\alpha}{2} \right)\right)^{-1}}\norm{w_n(t)-w_{n-1}(t)}_{\mathbb{U}}&\nonumber\\
&\le\frac{1}{4}\norm{w_n(t)-w_{n-1}(t)}_{\mathbb{U}},&
\end{flalign*}
provided that $ \norm{u_0}_{\D1} $ is sufficiently small. In addition, for any $ t<\mathscr{T}_{**} $, we note that 
\begin{align*}
\Bgnorm{\int_{0}^{t}\mathscr{R}(t-\tau)\big(H(w_{n}(\tau))-H(w_{n-1}(\tau))\big)\d\tau}_{\Lxi}\le\mathrm{I}_1\le\mathrm{I}_1+\mathrm{I}_2.
\end{align*}
Wherefore, for any $ t\in I $ the following estimate holds
\begin{align}\label{Sect5 Cauchy sp1}
\sup_{t\in I}\norm{w_{n+1}(t)-w_{n}(t)}_{\D1}&=\sup_{t\in I}\norm{\int_{0}^{t}\mathscr{R}(t-\tau)\big(H(w_{n}(\tau))-H(w_{n-1}(\tau))\big)\d\tau}_{\Lxi}\nonumber\\
&\le\mathrm{I}_1+\mathrm{I}_2\le\frac{1}{2}\norm{w_n(t)-w_{n-1}(t)}_{\mathbb{U}}.
\end{align}
On the other hand, by replacing respectively $ w,v $ by $ w_n,w_{n-1}\in\mathbb{U} $ in Lemma \ref{Exp_Contraction_estimates}, the following inequality holds
\begin{align*}
&t^{\frac{\alpha}{2}}\norm{\int_{0}^{t}\mathscr{R}(t-\tau)\big(H(w_n(\tau))-H(w_{n-1}(\tau))\big)\d\tau}_{\Lxi}\nonumber\\
\le&\frac{2\mathscr{C}_3(2\norm{u_0}_{\D1})^2}{\left(\beta\left(\alpha,1-\frac{3\alpha}{2} \right)\right)^{-1}}\left(\left(\Gamma\left(3\right)\right)^{\frac{1}{2}}+\left(\Gamma\left(4\right)\right)^{\frac{1}{2}}\left(6(2\norm{u_0}_{\D1})^2\right)^{\frac{1}{6}}\right)\norm{w_{n}(t)-w_{n-1}(t)}_{\mathbb{U}}.
\end{align*} 
As a result, if $ \norm{u_0}_{\D1} $ is small enough, we deduce
\begin{align}\label{Sect5 Cauchy sp2}
\sup_{t\in I\setminus \{0\}}t^{\frac{\alpha}{2}}\norm{w_{n+1}(t)-w_{n}(t)}_{\Lxi}\le\frac{1}{2}\norm{w-v}_{\mathbb{U}}.
\end{align}
On account of \eqref{Sect5_Invariant_sp_2}, \eqref{Sect5 Cauchy sp1} and \eqref{Sect5 Cauchy sp2}, we can conclude that
\begin{align*}
\norm{w_{n+1}(t)-w_{n}(t)}_{\mathbb{U}}\le\frac{1}{2}\norm{w_{n}(t)-w_{n-1}(t)}_{\Lxi},
\end{align*}
for every $ w_n,w_{n-1}\in\mathbb{U} $. It means that  $ \{w_n\}_{n=1}^\infty $ is a Cauchy sequence in $ \mathbb{U} $. The proof is completed.
\end{proof}
\begin{remark}
We note that in the above proof, we consider only the case $ \mathscr{T}>\max\left(\mathscr{T}_*,\mathscr{T}_{**}\right) $. When, $ \mathscr{T}<\min\left(\mathscr{T}_*,\mathscr{T}_{**}\right) $, the proof is similar and easier, so we omit it here.
\end{remark}

\section{Appendix}
\begin{definition}
Let $\alpha\in(0,1)$. Then, the definition of the M-Wright type function $\mathcal{W}_{\alpha}$ is given by
\begin{align*}
	\mathcal{W}_{\alpha}(r):=\sum_{k=0}^{\infty}\frac{r^{k}}{k!\Gamma(-\alpha k+1-\alpha)}.
\end{align*}
\end{definition}

\begin{lemma}\label{Property of M-Wright function}(\cite[Section 2]{fractional NavierStokes})
For $\alpha\in(0,1)$ and $\mu\in(-1, \infty)$, there holds
\begin{align*}
\int_{0}^{\infty}r^{\mu}\mathcal{W}_{\alpha}(r)\d r=\frac{\Gamma(1+\mu)}{\Gamma(1+\alpha\mu)}.
\end{align*}
\end{lemma}

\begin{lemma}\label{Fractional Sobolev embeddings}
Let $ \Omega\subset\R^d $ be a bounded domain with smooth boundary and $ p\in[1,\infty) $. Then, 
\begin{enumerate}[(i)]
\item if $~ 0\le \nu<\frac{d}{2} $ and $ 1\le p\le\frac{2d}{d-2\nu} $, or $ \nu=\frac{d}{2} $ and $ 1\le p<\infty $, we have
\begin{align*}
\Dnu\hookrightarrow L^p(\Omega);
\end{align*}
\item if $\frac{-d}{2}<\nu\le0 $ and $ p\ge\frac{2d}{d-2\nu} $, we have
\begin{align*}
\Dnu\hookleftarrow L^p(\Omega).
\end{align*}
\end{enumerate}
\end{lemma}

\begin{lemma}\label{Fractional Gronwall inequality}
(Fractional Gr\"onwall inequality) Let $ m,n$ be positive constants and $ \alpha\in(0,1). $ Suppose that function $ w\in L^{\infty}(0,T] $ satisfies the following inequality
\begin{align*}
w(t)\le m+n\int_{0}^{t}(t-\tau)^{\alpha-1}w(\tau)\d\tau,\quad\text{for all~}t\in(0,T],
\end{align*} 
then, the result below is satisfied
\begin{align*}
w(t)\le mE_{\alpha}\big(n\Gamma(\alpha)t^{\alpha}\big).
\end{align*}
\end{lemma}
\begin{definition}\label{Gamma-Beta defs}
Let $ p,q>0 $. Then, the Beta function and the Gamma function can be defined respectively by
\begin{align*}
\beta(p,q)=\int_{0}^{1}(1-z)^{p-1} z^{q-1}\d z\quad\text{ and }\quad\Gamma(p)=\int_{0}^{\infty}z^{p-1}e^{-z}\d z.
\end{align*}
\end{definition}
\section*{Acknowledgement}
The first author (Huy Tuan Nguyen) was supported by Vietnam National Foundation for Science and Technology Development (NAFOSTED) under grant number 101.02-2019.09.

\pagebreak
\end{document}